\documentclass[11pt]{amsart}
\usepackage{latexsym,amssymb,amsmath} 

\input epsf
\def\boxit#1{\vbox{\hrule height1pt\hbox{\vrule width1pt\kern3pt
  \vbox{\kern3pt#1\kern3pt}\kern3pt\vrule width1pt}\hrule height1pt}}

\def\trank{\text{rank}}

\def\BC{\mathbb C}

\def\BP{\mathbb P}
\def\pp#1{\mathbb P^{#1}}

\def\pp#1{{\mathbb P}^{#1}}

\def\hd{,\ldots,} 
\def\ww{\wedge}
\def\upperp{{}^\perp}

\def\be{\begin{equation}}
\def\ene{\end{equation}}
\def\aaa{{\mathbf a}}
\def\bbb{{\mathbf b}}

\def\mult{{\rm mult}}\def\Zeros{{\rm Zeros}}\def\tzeros{{\rm Zeros}}

\def\tmin{\operatorname{min}}

\def\cH{{\mathcal H}}

\def\11{\mathbf 1}

\def\l{\lambda}
\def\a{\alpha}

\def\t{\tau}

\def\s{\sigma}

\def\d{\delta}

\def\ra{{\mathord{\;\rightarrow\;}}}

\def\tdet{{\rm det}}\def\tlim{{\rm lim}\;}
\def\tperm{{\rm perm}}

\newtheorem{theorem}{Theorem}[section]
\newtheorem{proposition}[theorem]{Proposition}
\newtheorem{lemma}[theorem]{Lemma}
\newtheorem{corollary}[theorem]{Corollary}

\theoremstyle{definition}

\theoremstyle{remark}
\newtheorem{remark}[theorem]{Remark}

\def\ur{\underline{R}}

\def\dual{{^\vee}}

\def\intprod{\negthinspace
\mathbin{\raisebox{.4ex}{\hbox{\vrule height .5pt width 4pt depth 0pt %
          \vrule height 4pt width .5pt depth 0pt}}}} 

\def\Sym{S}
\DeclareMathOperator{\Lker}{Lker}
\DeclareMathOperator{\Rker}{Rker}
\DeclareMathOperator{\codim}{codim}
\DeclareMathOperator{\rank}{rank}

\DeclareMathOperator{\Sub}{Sub}

\newcommand{\defining}[1]{\textit{#1}}

\begin{document}
\title{On the ranks and border ranks of   symmetric tensors}
\author{J.M. Landsberg and Zach Teitler}
\date{September 25, 2009}
\begin{abstract}    Motivated by questions arising in signal processing, computational
complexity, and other areas, we   study   the
ranks  and border ranks  of symmetric tensors using geometric methods.
We provide improved lower bounds for the rank of a symmetric tensor (i.e., a homogeneous polynomial)
obtained by considering the   singularities of the hypersurface defined by the polynomial.
We   obtain normal forms for polynomials of border rank up to five,
and compute or bound the ranks of several classes of polynomials, including monomials,
the determinant, and the permanent. 
\end{abstract}
 \thanks{Landsberg supported by NSF grant DMS-0805782}
\email{jml@math.tamu.edu, zteitler@math.tamu.edu}
\maketitle
\centerline{ Communicated by Peter B\"urgisser }

\section{Introduction}\label{intro}
  Let
$S^d\BC^n$ denote the space of complex homogeneous polynomials of degree $d$ in $n$ variables. The \defining{rank} (or \defining{Waring rank})
$R(\phi)$ of a 
  polynomial $\phi\in S^d\BC^n$   is the smallest number $r$
such that $\phi$ is expressible as a sum of $r$ $d$-th powers, $\phi=x_1^d+\cdots +x_r^d$ with $x_j\in\BC^n$.
 The   \defining{border rank} $\ur(\phi)$ of $\phi$, is the smallest
$r$ such that $\phi$ is in the Zariski closure of the set of polynomials
of rank $r$ in $S^d\BC^n$, so in particular $R(\phi)\geq \ur(\phi)$. Although our perspective is geometric, we delay the introduction of geometric
language in order to first state our results in a manner more accessible to engineers and complexity theorists.

\smallskip

 Border ranks of polynomials
have been studied extensively, dating at least back to Terracini, although many   questions important
for applications to enginering and algebraic complexity theory   are still open. 
For example, in applications, one would like to be able to explicitly compute the ranks
and border ranks of polynomials. In the case of border rank, this
could be done if one had equations for the variety of polynomials of border rank $r$.
Some equations have been known for nearly a hundred years: given a polynomial  $\phi\in S^d\BC^n$,
we may {\it polarize}  it and consider it as a multi-linear form $\tilde \phi$, where
$\phi(x)=\tilde\phi(x\hd x)$. We can then feed $\tilde\phi$ $s$ vectors, to consider it as
a linear map $\phi_{s,d-s}: S^s\BC^{n*}\ra S^{d-s}\BC^n$, where
$\phi_{s,d-s}(x_1\cdots x_s)(y_1\cdots y_{d-s}) = \tilde\phi(x_1,\dots,x_s,y_1,\dots,y_{d-s})$.
Then for all $1\leq s\leq d$,    
\begin{equation}\label{someeqns}  
 \ur(\phi)\geq \trank\phi_{s,d-s}.
\end{equation} 
This follows immediately from, for example, the inverse systems of Macaulay~\cite{MR1281612}.
See Remark~\ref{Oldboundproof} for a proof.
These equations are sometimes called {\it minors of Catalecticant matrices}
or {\it minors of symmetric flattenings}.

One important class of polynomials in applications are the monomials. We apply the above equations, combined with 
techniques from differential geometry to prove:

\begin{theorem}\label{monomialthm} Let $a_0\hd a_m$ be non-negative integers satisfying
$a_0 \geq a_1+\cdots + a_m$. Then
\[
  \ur(x_0^{a_0} x_1^{a_1} \cdots x_m^{a_m}) = \prod_{i=1}^{m} (1+a_i) .
\]
\end{theorem}
For other monomials we give upper and lower bounds on the border rank, see Theorem \ref{urmonomialest}.

We also use   differential-geometric methods to determine normal forms for polynomials of border rank
at most five  and estimate their ranks, Theorems \ref{normalformsprop}, \ref{normalformsprop4}, \ref{normalformsprop5}. For example:

\begin{theorem}\label{quickvers}  The polynomials of border rank three have the following normal forms:
\[
\begin{array}{ |c|c|}
\hline
  {\rm normal\ form}&  R \\
  x^d+y^d+z^d         & 3\\
  x^{d-1}y+z^d         & d \leq R \leq d+1\\
    x^{d-2}y^2+x^{d-1}z & d \leq R\leq 2d-1\\
\hline
\end{array}
\] 
\end{theorem}
  Here one must account for
the additional cases where $x,y,z$ are linearly dependent, but in these cases one
can normalize, e.g. $z=x+y$.
More information is given  in Theorem \ref{normalformsprop}.

\smallskip

To obtain new bounds on rank, we use algebraic geometry, more specifically the
  singularities of the hypersurface determined by a polynomial
$\phi$. Let $\tzeros (\phi)=\{ [x]\in \BP \BC^{n*}\mid \phi(x)=0\}\subset \BP \BC^{n*}$ denote the zero set
of $\phi$. Let $x_1\hd x_n$ be linear coordinates on $\BC^{n*}$ and
define
\[
  \Sigma_s(\phi) :=
    \left\{ [x] \in \tzeros(\phi)
    \left| \, \frac{\partial^I\phi}{\partial x^I}(x)=0 , \forall |I|\leq s \right\}\right.
\]
so $\Sigma_0(\phi)=\tzeros (\phi)$ and $\Sigma_1(\phi)$ is the set of singular points
of $\tzeros(\phi)$.

While the following result is quite modest, we remark that it is the first new general
lower bound on rank that we are aware of in about 100 years (since the bound \eqref{someeqns}):

\begin{theorem}\label{sigmathm}
Let $\phi\in S^d\BC^n$ with $\langle\phi\rangle=\BC^n$.
Let $1 \leq s \leq d$. Use the convention that $\dim \emptyset=-1$.
Then,
\[
  R(\phi) \geq \rank \phi_{s,d-s} + \dim \Sigma_{s}(\phi) + 1.
\]
\end{theorem}

The right hand side of the inequality is   typically maximized at $s = \lfloor d/2 \rfloor$, see \S\ref{Stanleyrem}.

For example, applying Theorem \ref{sigmathm} to the determinant and permanent polynomials
(see \S\ref{detpersect}) yields

\begin{corollary}
\begin{align*}
&R(\tdet_n)\geq \binom n{\lfloor n/2 \rfloor}^2 + n^2-(\lfloor n/2 \rfloor+1)^2 , \\
&R(\tperm_n)\geq \binom n{\lfloor n/2 \rfloor}^2 + n(n-\lfloor n/2 \rfloor-1) .
\end{align*}
\end{corollary}
  Gurvits \cite{Gurvits} had previously observed
$\ur(\tdet_n)\geq \binom n{\lfloor n/2 \rfloor}^2$ and $\ur(\tperm_n)\geq \binom n{\lfloor n/2 \rfloor}^2$
by using \eqref{someeqns}.

\medskip

We expect that further study of singularities will produce  significantly stronger general
lower bounds for rank, including bounds that involve the degree as well as the number of variables.

\smallskip

As a consequence of  our study of rank in a more general geometric context, we prove
\begin{corollary}\label{introcor}
Given $\phi\in S^d\BC^n$,
$R(\phi)\leq \binom{n+d-1}d-n+1$,
\end{corollary}\noindent
which is a corollary of Proposition  \ref{anyvarubound}.
(The bound $R(\phi)\leq \binom{n+d-1}d$
is trivial, as we explain in \S\ref{arbvarietiessect}.)

\begin{remark}
Schinzel studies similar questions for polynomials
over arbitrary fields in~\cite{MR2040699,MR1899238}.
Since he is concerned with not-necessarily
homogeneous polynomials, the nature of the results are
slightly different than the results here.
The set of representations of a polynomial as a sum of powers
is studied geometrically in~\cite{MR1201387} and \cite{MR1780430}.
\end{remark}

\subsection{Overview}
We begin in \S\ref{geomsect} by phrasing the problems in geometric language.
We then review standard facts about rank and border rank in \S\ref{reviewsect}.
In \S\ref{cssect} we give an exposition of a
theorem of Comas and Seiguer \cite{CS}, which completely describes the possible
ranks of homogeneous polynomials in two variables.
We then discuss ranks for arbitrary varieties and prove  Proposition \ref{anyvarubound}, which gives an upper bound for
rank valid for an arbitrary variety in \S\ref{arbvarietiessect}. Applying Proposition \ref{anyvarubound} to polynomials yields
  Corollary  \ref{introcor}   above.
In \S\ref{lowerboundsect} we prove Theorem \ref{sigmathm}.
We then study some specific cubic polynomials in an arbitrary number of variables in \S\ref{cubicsect}.
In \S\ref{cubiccurvesect} we give a   presentation of the possible ranks, border ranks
and normal forms of degree three polynomials in three variables  that slightly
refines   the presentation in  \cite{comonmour96}.
This is followed by a brief discussion of bounds on rank and border rank for  determinants and permanents
in \S\ref{detpersect}.
In \S\ref{verlimits} we return to a general study of
limiting secant planes and use this to classify polynomials of border ranks up to five.
  We conclude with
a study of the ranks and border ranks of monomials in \S\ref{monomialsect}.

\smallskip

\subsection*{Acknowledgements} This paper grew out of questions raised at the
2008 AIM workshop {\it Geometry and representation theory of tensors
for computer science, statistics and other areas}, and the authors
thank AIM and the conference participants for inspiration.
We also thank
L.~Matusevich for furnishing Proposition \ref{monomiallowerbound} and
 B.~Reznick for several comments, including 
a suggestion
related to Example~\ref{example x2yz}, and also for
providing an unpublished
manuscript related to Theorem~\ref{cubicranks}.

\section{Geometric definitions}\label{geomsect}

Definitions of rank and border rank in a more general context are as follows:
Let $V=\BC^n$ denote a complex vector space and let $\BP V$ denote
the associated projective space. For a subset $Z\subset \BP V$,
we let $\langle Z\rangle \subseteq V$ denote its linear span.
For a variety $X\subset \BP V$, define 
\begin{equation}
\s_r^0(X) = \left\{ \bigcup_{x_1\hd x_r\in X}\BP \langle x_1\hd x_r\rangle \right\} 
\subset \BP V, \qquad
\s_r(X)     = \overline{\left\{\bigcup_{x_1\hd x_r\in X}\BP \langle x_1\hd x_r\rangle\right\}}
\subset \BP V
\end{equation}
where the overline denotes Zariski closure. These are respectively
the points that lie on some secant $\pp{r-1}$ to $X$ and the Zariski
closure of the set of such points, called the \defining{variety of secant
$\pp {r-1}$'s to $X$}. For
  $p\in \BP V$, define the \defining{$X$-rank} of $p$, $R_X(p):=\{ \min r \mid  p\in \s_r^0(X)\}$
and the \defining{$X$-border rank} of $p$, $\ur_X(p):=\{ \min r \mid p \in \s_r (X)\}$.
In geometry, it is more natural to study   border rank  than rank, because by definition
the set of points of border rank at most $r$ is an algebraic variety.
Let $S^dW$ denote the space of homogeneous polynomials of degree $d$ on $W^*$ and  let  $v_d(\BP W)\subset \BP (S^dW)$ denote the \defining{Veronese variety}, the (projectivization of the) 
set of $d$-th powers.  
Then, comparing with the definitions of \S\ref{intro},
$R(\phi)=R_{v_d(\BP W)}([\phi])$ and $\ur(\phi)=\ur_{v_d(\BP W)}([\phi])$. Advantages
of the more general definitions include that it is
often easier to prove statements in the context of an arbitrary variety, and that
one can simultaneously study the ranks of polynomials and   tensors
(as well as other related objects).
We also let $\t(X)\subset \BP V$ denote the variety of embedded tangent $\pp 1$'s to $X$, called
the {\it tangential variety of $X$} and note that $\t(X)\subseteq \s_2(X)$.

At first glance, the set of polynomials (respectively points in $\BP V$) of a given rank (resp. $X$-rank) appears to lack interesting
geometric structure---it can have
components of varying dimensions and fail to be a closed projective variety.   One principle
of this paper  is that {\it among polynomials of a given border rank, say $r_0$, the polynomials having rank
greater than $r_0$ can be distinguished by their singularities.}
For a hypersurface $X\subset \BP V$ and $x\in X$, define $\mult_x(X)$ to be
the order of vanishing of the defining equation for $X$ at $x$.

Consider the following stratification of $\BP S^dW$. Let
\[
  {v_d(\BP W^*)_k}\dual := \BP \{ \phi\in S^dW\mid \exists [p]\in \tzeros(\phi),\, \mult_{[p]}(\tzeros(\phi))\geq k+1 \} .
\]
Then
\[
\BP S^d W={v_d(\BP W^*)_0}\dual \supset v_d(\BP W^*)\dual= {v_d(\BP W^*)_1}\dual
\supset \cdots \supset {v_d(\BP W^*)_{d}}\dual = \emptyset .
\]
Among polynomials of a given border rank, we expect the deeper they lie in this stratification, the
higher their rank will be.
(The analogous  stratification of $\BP V^*$ can be defined for arbitrary varieties $X\subset \BP V$.
It begins with $\BP V^*$ and the next stratum is $X\dual$.)
A first step in this direction is Theorem \ref{sigmathm}.
We expect the general study of points whose $X$-rank is greater than their $X$-border rank will be closely
related to stratifications of dual varieties.

\section{Review of known facts about rank and border rank of polynomials}\label{reviewsect}

\subsection{The Alexander-Hirschowitz Theorem}
The expected dimension of $\s_r(X^n)\subset \pp N$ is $\tmin\{r(n+1)-1, N\}$, and if 
$\s_r(X)$ fails to have this expected dimension it is called \defining{degenerate}. 
  Alexander and Hirschowitz \cite{AH}, building on work of Terracini,
showed that the varieties $\s_r(v_d(\BP W))$ are all of the expected dimensions
with a short, well understood, list of exceptions, thus the rank and border rank
of a generic polynomial of degree $d$ in $n$ variables is known for all $d,n$. (Note that  
it is essential to be working over an algebraically closed field to talk about a generic polynomial.)  See \cite{Ottwaring} for an excellent exposition of the Alexander-Hirschowitz theorem.
 
\subsection{Subspace  varieties}
Given $\phi\in S^dW$,   define the \defining{span} of $\phi$ to be
$\langle \phi\rangle =\{ \a\in W^*\mid \a\intprod\phi=0\}\upperp\subset W$,
where $\a \intprod \phi = \partial \phi / \partial \alpha$ is the partial derivative of $\phi$ by $\alpha$.
Then $\dim \langle \phi\rangle$ is the minimal number of variables
needed to express $\phi$ in some coordinate system and
  $\phi \in S^d \langle \phi \rangle$.
If $\langle \phi \rangle \neq W$ then the vanishing set $\Zeros(\phi) \subset \BP W^{*}$
is a cone over $\{ [\alpha] \mid \alpha \intprod \phi = 0 \}$.

For $\phi \in S^dW$,
\begin{align}
 \label{subspaceeqn}
&R_{v_d(\BP W)}(\phi)=R_{v_d(\BP \langle\phi\rangle)}(\phi) ,
\\
&\label{rsubspaceeqn}
\ur_{v_d(\BP W)}(\phi)=\ur_{v_d(\BP \langle\phi\rangle)}(\phi) ,
\end{align}
see, e.g., \cite{LMor} or \cite[Prop. 3.1]{Limdesilva}.

Define the \defining{subspace variety}
\[
  \Sub_k=\BP \{\phi\in S^dW \mid \dim \langle \phi\rangle\leq k\}.
\]
Defining equations of $\Sub_k$ are given by the $(k+1)\times (k+1)$ minors
of $\phi_{1,d-1}$ (see, e.g, \cite[\S 7.2]{weyman}), so in particular
\begin{equation}
  \label{subprop}\s_k(v_d(\BP W))\subseteq \Sub_k ,
\end{equation}
 i.e.,
$[\phi]\in \s_r(v_d(\BP W))$ implies $\dim \langle \phi\rangle\leq r$.
We will often work  by induction and assume $\langle \phi\rangle = W$.
In particular,  we often restrict attention to  
$\s_r(v_d(\BP W))$ for $r\geq \dim W$.

\subsection{Specialization}\label{subsect: specialization}
If $X\subset \BP V$ is a variety and we consider the image of the cone
$\hat X\subset V$ under a projection $\pi_U:   V\to  (V/U)$
where $U \subset   V$ is a subspace,
then  for $p\in \BP V$, $R_{\pi_U(X)}(\pi_U(p))\leq
R_{X}(p)$ and similarly for border rank. To see this, if
$p = q_1+\cdots + q_r$, then $\pi_U(p)=\pi_U(q_1)+\cdots + \pi_U(q_r)$
because $\pi_U$ is a linear map.
In particular, given a polynomial in $n+m$ variables,  $\phi(x_1, \dots, x_n , y_1, \dots, y_m)$,
if we  set the $y_i$ to be linear combinations
of the $x_j$, then
\begin{equation}
  \label{specialeqn}R_{v_d(\BP \BC^{n})}(\phi(x,y(x)))\leq R_{v_d(\BP \BC^{n+m})}(\phi(x,y)).
\end{equation}

\subsection{Symmetric Flattenings (Catalecticant minors)}\label{basicfacts}

For $r < \frac{1}{n} \binom {n+d-1}{n-1}$, some equations for $\s_r(v_d(\BP W))$ are known, but not enough
to generate the ideal in most cases. The main known equations come from {\it symmetric
flattenings}, also known as {\it catalecticant matrices},
as described in~\eqref{someeqns}.
Other equations are discussed in \cite{ottcubic}, and there is recent
work describing general methods for obtaining further equations, see \cite{LO}.
Here are the symmetric flattenings:

For $ \phi \in S^dW$, define the contracted maps
\begin{equation}\label{phids}
\phi_{s,d-s}: S^{s}W^{*} \times S^{d-s}W^{*} \to \BC.
\end{equation}
Then we may consider the left and right kernels $\Lker \phi_{s,d-s} \subseteq S^{s}W^{*}$,
$\Rker \phi_{s,d-s} \subseteq S^{d-s} W^{*}$.
We will abuse notation and identify $\phi_{s,d-s}$ with the associated
map $S^{d-s}W^{*} \to   S^{s}W$. We restrict attention to $\phi_{s,d-s}$ for $1 \leq s \leq \lfloor d/2 \rfloor$ to avoid redundancies.

\begin{remark}\label{Stanleyrem}
The sequence $\{ \rank \phi_{s,d-s} : 1 \leq s \leq \lfloor \frac{d}{2} \rfloor \}$ may decrease,
as observed by Stanley~\cite[Example 4.3]{MR0485835}.
For instance, let
\[
\begin{split}
  \phi &= x_1 x_{11}^3
            + x_2 x_{11}^2 x_{12}
            + x_3 x_{11}^2 x_{13}
            + x_4 x_{11} x_{12}^2
            + x_5 x_{11} x_{12} x_{13} \\
         & \quad
            + x_6 x_{11} x_{13}^2
            + x_7 x_{12}^3
            + x_8 x_{12}^2 x_{13}
            + x_9 x_{12} x_{13}^2
            + x_{10} x_{13}^3 .
\end{split}
\]
Then $\rank \phi_{1,3} = 13$ but $\rank \phi_{2,2} = 12$.
On the other hand, Stanley showed that if $\dim W \leq 3$
and $\phi \in S^d W$, then $\rank \phi_{s,d-s}$ is nondecreasing
in $1 \leq s \leq \lfloor \frac{d}{2} \rfloor$ \cite[Theorem 4.2]{MR0485835}.
\end{remark}
 
\begin{remark} When $\dim W=2$, for any value of $s$ such that $s,d-s\geq r+1$,
one obtains a set of generators for $I(\s_r(v_d(\BP W))$ by the $r+1$ by $r+1$ minors
of the $s,d-s$ symmetric flattening.
This was known by Sylvester; see also \cite{MR1735271}.
Also when $r=2$, taking $s=1$ and $s=2$ is enough to obtain generators
of $I(\s_2(v_d(\BP W)))$, see \cite{MR1703911}.
The hypersurface $\s_3(v_3(\pp 2))$ is given by
a degree four equation called the {\it Aronhold invariant}, which does not arise as a symmetric flattening
(see \cite{ottcubic}).
  Very few other
cases are understood. 
\end{remark}

\subsection{A classical lower bound for rank}
The following is a symmetric analog of a result that is well known for tensors, e.g. \cite[Prop.\ 14.45]{BCS}.
\begin{proposition}\label{spanbound} $R(\phi)$ is at least the minimal number of elements of
$v_{s}(\BP W)$ needed to span (a space containing) $\BP (\phi_{s,d-s}(S^{d-s}W^*))$.
\end{proposition}
\begin{proof}
If $\phi=\eta_1^d+\cdots + \eta_r^d$, then
$\phi_{s,d-s}(S^{d-s}W^*) \subseteq \langle \eta_1^{s}, \dots , \eta_r^{s}\rangle$.
\end{proof}

\subsection{Spaces of polynomials where the possible ranks and border ranks are known}
The only cases are as follows. (i.) $S^2\BC^n$ for all $n$. Here rank and border rank coincide with the rank of
the corresponding symmetric matrix, and there is a normal form for elements of rank $r$,
namely $x_1^2+\dots+ x_r^2$. (ii.) $S^d\BC^2$ where the possible ranks and border ranks are known,
see Theorem \ref{csmainthm}. However there are no   normal forms in general.
(iii.) $S^3\BC^3$ where the possible ranks and border ranks were determined
in \cite{comonmour96}. We also explicitly describe which normal forms have which ranks in \S\ref{cubicsect}.
The normal forms date back to \cite{MR1506892}.

\section{The theorem of Comas and Seiguer}\label{cssect}

\begin{theorem}[Comas-Seiguer, \cite{CS}]\label{csmainthm}
Consider $v_d(\pp 1)\subset \pp d$, and recall that $\s_{\lfloor \frac{d+1}{2}\rfloor}(v_d(\pp 1))=\pp d$.
Let $r\leq \lfloor \frac{d+1}{2}\rfloor$.
Then
\[ \s_r(v_d(\pp 1)) = \{ [\phi] : R (\phi) \leq r \} \cup \{ [\phi] : R (\phi) \geq d-r+2 \} . \]
\end{theorem}
 By Proposition~\ref{anyvarubound},
$R(\phi) \leq d$ for all $\phi$.
Hence the above statement is equivalent to the following:
\[
\s_r(v_d(\pp 1)) \setminus \s_{r-1}(v_d(\pp 1)) = \{ [\phi] : R(\phi) = r \} \cup \{ [\phi] : R(\phi) = d-r+2 \} .
\]

Throughout this section  we  write $W=\BC^2$.

\begin{lemma}\label{lemma:rank-bound-dual}
Let $\phi \in \Sym^d(W)$.
Let $1 \leq r \leq d-1$.
Then $R (\phi) > r$ if and only if $\BP \Lker\phi_{r,d-r}  \subset v_r(\BP W)\dual$.
\end{lemma}
Recall that for any $W$, $v_r(\BP W)^\dual$ is the set of singular hypersurfaces of degree $r$ in $\BP W^*$;
for $W=\BC^2$ this is the set of polynomials on $\BP^1$ with a multiple root.
\begin{proof}
First say$R(\phi)\leq r$ and write $\phi=w_1^d+\cdots +w_r^d$. Then
$\Lker\phi_{r,d-r}$ contains the polynomial
with distinct roots $w_1\hd w_r$.
Conversely, say $0 \neq P \in \Lker\phi_{r,d-r}$ has distinct roots $w_1\hd w_r$.
It will be sufficient to show $\phi\ww   w_1^d\ww \dots \ww  w_r^d=0$.
  We show
$\phi\ww   w_1^d\ww \dots \ww  w_r^d(p_1\hd p_{r+1})=0$
for all $p_1\hd p_{r+1}\in \Sym^dW^*$ to finish the proof.
Rewrite this as
\[
\phi(p_1)m_1-\phi(p_2)m_2+\cdots +(-1)^{r}\phi(p_{r+1})m_{r+1}
=
\phi(m_1p_1+\cdots +(-1)^{r}m_{r+1}p_{r+1})
\]
where $m_j=w_1^d\ww\cdots\ww w_r^d(p_1\hd \hat p_j\hd p_{r+1}) \in \BC$ 
(considering $\Sym^dW$ as the dual
vector space to $\Sym^dW^*$).
Now for each $j$, 
\[
\begin{split}
  w_j^d (m_1p_1+\cdots +(-1)^{r}m_{r+1}p_{r+1})
    &= \sum_{i=1}^{r+1} w_j^d ( (-1)^{i-1} m_i p_i ) \\
    &= \sum_{i=1}^{r+1} (-1)^{2(i-1)} w_j^d \wedge w_1^d \wedge \cdots \wedge w_{r}^d (p_1\hd p_{r+1}) \\
    &= 0.
\end{split}
\] 
Hence, now considering
the $p_j$ as polynomials of degree $d$ on $W$,
\[ (m_1p_1+\cdots +(-1)^{r}m_{r+1}p_{r+1})(w_i)=0 \]
for each $i$.
But then  $ (m_1p_1+\cdots +(-1)^{r}m_{r+1}p_{r+1})=PQ$
for some $Q\in \Sym^{d-r}W^*$ and $\phi(PQ)=0$ because
  $P\in \Lker \phi_{r,d-r}$.
\end{proof}

As mentioned above, the generators of the ideal of $\s_r(v_d(\pp 1))$ can be obtained from  the $(r+1)\times (r+1)$ minors
of   $\phi_{s,d-s}$. Thus (see \cite{Geramita} for more details):

\begin{lemma}\label{lemma:codim-lker} 
For $\phi\in S^d\BC^2$ and   $1 \leq r \leq \lfloor d/2 \rfloor$ the following are equivalent.
\begin{enumerate}
\item $[\phi] \in \s_r(v_d(\pp 1))$,
\item $\rank \phi_{s,d-s} \leq r$ for $s = \lfloor d/2 \rfloor$,
\item $\rank \phi_{r,d-r} \leq r$,
\item $\Lker \phi_{r,d-r} \neq \{ 0 \}$.
\end{enumerate}
\end{lemma}

\begin{lemma}
Let $r\leq \lfloor \frac{d+1}{2}\rfloor$.
If $\phi = \eta_1^d + \dots + \eta_k^d$, $k \leq d-r+1$,
and $P \in \Lker \phi_{r,d-r}$,
then $P(\eta_i)=0$ for each $1 \leq i \leq k$.
\end{lemma}

\begin{proof}
For $1 \leq i \leq k$ let $M_i \in W^{*}$ annihilate $\eta_i$.
In particular, $M_i(\eta_j) \neq 0$ if $j\neq i$, because the $[\eta_j]$ are distinct:
$\eta_j$ is not a multiple of $\eta_i$.
Let $L \in W^{*}$ not vanish at any $\eta_i$.
For each $i$, let
\[ g_i = P M_1 \cdots \widehat{M}_i \cdots M_k L^{d-r+1-k} , \]
so $\deg g_i = d$.
Since $P \in \Lker \phi_{r,d-r}$ we get $\phi(g_i)=0$.
On the other hand, $\eta_j^d(g_i) = 0$ for $j \neq i$,
so
\[ \eta_i^d(g_i) = 0 = P(\eta_i) M_1(\eta_i) \cdots \widehat{M_i(\eta_i)} \cdots M_k(\eta_i) L(\eta_i)^{d-r+1-k} . \]
All the factors on the right are nonzero except possibly $P(\eta_i)$.
Thus $P(\eta_i)=0$.
\end{proof}

\begin{proof}[Proof of Theorem \ref{csmainthm}]
Suppose $[\phi] \in \s_r(v_d(\pp 1))$ and $R(\phi) \leq d-r+1$.
Write $\phi = \eta_1^d + \dots + \eta_k^d$ for some $k \leq d-r+1$
and the $[\eta_i]$ distinct.
  $[\phi] \in \s_r(v_d(\pp 1))$ implies $\rank \phi_{r,d-r} \leq r$,
so $\dim \Lker \phi_{r,d-r} \geq 1$.
Therefore there is some nonzero $P \in \Lker \phi_{r,d-r}$.
Every $[\eta_i]$ is a zero of $P$, but $\deg P = r$
so  $P$ has at most $r$ roots.
So in fact $k \leq r$.
This shows the inclusion $\subseteq$ in the statement of the theorem.

We must show $\{ [\phi] : R(\phi) \geq d-r+2 \} \subseteq \s_r(v_d(\pp 1))$.
For $r=1$, the first set is empty, since each polynomial $\phi$ has rank at most $d$
by Proposition~\ref{anyvarubound}.
So suppose $r>1$, $R(\phi) \geq d-r+2$,
and $[\phi] \notin \s_{r-1}(v_d(\pp 1))$.
Then $\codim \Rker \phi_{r-1,d-r+1} = r$
by Lemma~\ref{lemma:codim-lker},
and $\BP \Rker \phi_{r-1,d-r+1} \subset v_r(\BP W)\dual$ by Lemma~\ref{lemma:rank-bound-dual}
(applied to $\Rker \phi_{r-1,d-r+1} = \Lker \phi_{d-r+1,r-1}$).
This means every polynomial $P \in \Rker \phi_{r-1,d-r+1}$ has a singularity (multiple root in $\pp 1$).
By Bertini's theorem, there is a basepoint of the linear system
(a common divisor of all the polynomials in $\Rker \phi_{r-1,d-r+1}$).
Let $F$ be the greatest common divisor.
Say $\deg F = f$.
Let $M = \{ P/F \mid P \in \Rker \phi_{r-1,d-r+1} \}$.
Every $P/F \in M$ has degree $d-r+1-f$.
So $\BP M \subset \BP \Sym^{d-r+1-f} W^{*}$, which has dimension $d-r+1-f$.
Also $\dim \BP M = \dim \BP\Rker \phi_{r-1,d-r+1} = d-2r+1$.
Therefore $d-2r+1 \leq d-r+1-f$, so $f \leq r$.

Since the polynomials in $M$ have no common roots,  
$(\Sym^{r-f}W^{*}).M = \Sym^{d-2f+1}W^{*}$ (see, e.g. \cite{harris}, Lemma~9.8).
Thus
\[
  \Sym^{r-1}W^{*}.\Rker \phi_{r-1,d-r+1} = \Sym^{f-1}W^{*}.\Sym^{r-f}W^{*}.M.F = \Sym^{d-f}W^{*}.F .
\]
So if $Q\in S^{d-f}W^*$,
then $FQ = GP$ for some $G \in \Sym^{r-1}W^{*}$ and $P \in \Rker \phi_{r-1,d-r+1}$,
so $\phi(FQ) = \phi(GP) = 0$.
Thus $0 \neq F \in \Lker \phi_{f,d-f}$, so $[\phi] \in \s_f(v_d(\pp 1))$.
And finally $\s_f(v_d(\pp 1)) \subset \s_r(v_d(\pp 1))$, since $f \leq r$.
\end{proof}

\begin{corollary}\label{corab}
If $a, b > 0$ then $R(x^a y^b) = \max(a+1, b+1)$.
\end{corollary}
\begin{proof}
Assume $a \leq b$.
The symmetric flattening $(x^a y^b)_{a,b}$ has rank $a+1$
(the image is spanned by $x^a y^0,x^{a-1}y^1, \dots, x^0 y^a$);
it follows that $\ur(x^a y^b) \geq a+1$.
Similarly, $(x^a y^b)_{a+1,b-1}$ has rank $a+1$ as well.
Therefore $\ur(x^a y^b) = a+1$, so $R(x^a y^b)$ is either $b+1$ or $a+1$.

Let $\{\alpha, \beta\}$ be a dual basis to $\{x,y\}$.
If $a < b$ then $\BP \Lker (x^a y^b)_{a+1,b-1} = \{ [ \alpha^{a+1} ] \} \subset v_{a+1}(\BP W)\dual$.
Therefore $R(x^a y^b) > a+1$.
If $a = b$ then $R(x^a y^b) = a+1 = b+1$.
\end{proof}

In particular, $R(x^{n-1}y)=n$.

\section{Maximum rank of arbitrary varieties}\label{arbvarietiessect}

For any variety $X\subset \BP V=\pp N$ that is not contained in
a hyperplane,  {\it a priori} the maximum $X$-rank of any point is
$N+1$ as we may take a basis of $V$ consisting of elements of
$X$. This maximum occurs if, e.g., $X$ is a collection of $N+1$ points.

\begin{proposition}\label{anyvarubound}
Let $X\subset \pp N=\BP V$ be an irreducible variety of dimension $n$ not contained in
a hyperplane. Then for all $p\in \BP V$, $R_X(p)\leq N+1-n$.
\end{proposition}

\begin{proof}
If $p \in X$ then $R_X(p) = 1 \leq N+1-n$.
Henceforth we consider only $p \notin X$.
Let $\cH_p$ be the set of hyperplanes containing $p$.

We proceed by induction on the dimension of $X$.
If $\dim X = 1$, for a general  $M \in \cH_p$,
$M$ intersects $X$ transversely by Bertini's theorem.
We claim $M$ is spanned by $M \cap X$.
Otherwise, if $M'$ is any other hyperplane containing $M \cap X$,
say $M$ and $M'$ are defined by linear forms $L$ and $L'$, respectively.
Then $L'/L$ defines a meromorphic function on $X$ with no poles,
since each zero of $L$ is simple and is also a zero of $L'$.
So $L'/L$ is actually a holomorphic function on $X$,
and since $X$ is projective, $L'/L$ must be constant.
This shows $M = M'$.
Therefore $M \cap X$ indeed spans $M$.

As noted above, by taking a basis for $M$ of points of $M \cap X$, we get
\[
  R_X(p) \leq R_{M \cap X}(p) \leq \dim M + 1,
\]
where $\dim M+1 = N+1-n$ since $n=1$ and $\dim M = N-1$.

For the inductive step,  define $\cH_p$ as above.
For general $M\in \cH_p$,  $M \cap X$ spans $M$ by the same argument,
and is also irreducible if $\dim X = n > 1$; see~\cite[pg.\ 174]{GH2}.
Note that $\dim M \cap X = n-1$ and $\dim M = N-1$.
Thus by induction, $R_{M \cap X}(p) \leq (N-1)+1-(n-1) = N+1-n$.
Since $M \cap X \subset X$ we have $R_X(p) \leq R_{M \cap X}(p) \leq N+1-n$.
\end{proof}

In particular:
\begin{corollary}\label{polyrkbound} Given $\phi\in S^d\BC^n$,
$R(\phi)\leq \binom{n+d-1}d-n+1$.
\end{corollary}

\begin{corollary} Let $C\subset \pp N=\BP V$ be a smooth curve not contained in a hyperplane.
Then the maximum $C$-rank of any $p\in \BP V$ is at most $N$.
\end{corollary}

We may refine the above discussion to ask, what is the maximum $X$-rank
of a point lying on a given secant variety of $X$, that is, with a bounded $X$-border rank?
For any $X$, essentially by definition,
$\{x\in \s_2(X)\mid R_X(x)>2\}\subseteq\t (X)\backslash X$.
The rank of a point
on $\t (X)$ can already be the maximum, as well as being arbitrarily large. Both these
occur   for
$X$ a rational normal curve of degree $d$ (see \S\ref{cssect}) where the rank of a point on
$\t(X)$ is the maximum $d$.

\section{Proof   and variants of Theorem \ref{sigmathm}}\label{lowerboundsect}

For $\phi \in S^d W$ and $s \geq 0$,
let 
\[
  \Sigma_{s}(\phi)=\Sigma_s:=\{ [\a]\in \Zeros(\phi) \mid \mult_{[\a]}(\phi) \geq s+1\} \subset \BP W^{*}.
\]
This definition agrees with our coordinate definition in \S\ref{intro}.

\begin{remark}
Note that for $\phi \in S^d W$, $\Sigma_d=\emptyset$ and $\Sigma_{d-1} = \BP \langle \phi \rangle^{\perp}$.
In particular,  $\Sigma_{d-1}$ is empty if and only if $\langle \phi \rangle = W$.
\end{remark}

\begin{remark}
The stratification mentioned in the introduction is identified as
\[
  {v_d(\BP W^*)_k}\dual
  = \BP \{ \phi \mid \Sigma_{k-1}(\phi) \neq \emptyset \} .
\]
It is natural to refine this stratification by the geometry of $\Sigma_{k-1}$, for example by:
\[
  {v_d(\BP W^*)_{k,a}}\dual :=
    \BP \{ \phi \mid \dim \Sigma_{k-1}(\phi) \geq a \} .
\]
\end{remark}

\begin{proposition}\label{prop: kernel meeting veronese}
\[
  v_{d-s}(\Sigma_s)=\BP \Rker\phi_{s,d-s}\cap v_{d-s}(\BP W^*).
\]
That is, $[\alpha] \in \Sigma_s$ if and only if $[\alpha^{d-s}] \in \BP \Rker \phi_{s,d-s}$.
\end{proposition}
\begin{proof}
For all $\a \in W^*$ and $w_1, \dots, w_s \in W^*$,  
\[
  \tilde\phi(w_1, \dots, w_s, \a, \dots, \a) = \left( \frac{\partial^s \phi}{\partial w_1 \cdots \partial w_s} \right) (\a) .
\]
Now $\a^{d-s} \in \Rker \phi_{s,d-s}$ if and only if the left hand side vanishes for all $w_1,\dots,w_s$,
and $\mult_{[\a]} \phi \geq s+1$ if and only if the right hand side vanishes for all $w_1,\dots,w_s$.
\end{proof}

\begin{lemma}\label{Lveremptylem}
Let $\phi \in S^d W$. Suppose we have an expression $\phi=\eta_1^d+\cdots + \eta_r^d$.
Let
$L := \BP \{ p\in S^{d-s}W^* \mid p( \eta_i )=0, 1\leq i\leq r\}$.
Then
\begin{enumerate}
\item $L \subset \BP \Rker \phi_{s,d-s}$.
\item $\codim L \leq r$.
\item If $\langle \phi \rangle = W$, then  $L\cap v_{d-s}(\BP W^*)=\emptyset$.
\end{enumerate}
\end{lemma}
\begin{proof}
For the first statement, for $p \in S^{d-s}W^*$ and any $q \in S^sW^*$,  
\[
  \phi_{s,d-s}(q)(p) = q(\eta_1)p(\eta_1) + \dots + q(\eta_r)p(\eta_r).
\]
If $[p] \in L$ then each $p(\eta_i)=0$, so $\phi_{s,d-s}(q)(p)=0$ for all $q$.
Therefore $p \in \Rker \phi_{s,d-s}$.

The second statement is well-known.
Since each point $[\eta_i]$ imposes a single linear condition on the coefficients of $p$,
$L$ is the common zero locus of a system of $r$ linear equations.
Therefore $\codim L \leq r$.

If $\langle \phi \rangle = W$,
then $W = \langle \phi \rangle \subseteq \langle \eta_1, \dots, \eta_r \rangle \subseteq W$,
so the $\eta_i$ span $W$.
Therefore the points $[\eta_i]$ in $\BP W$ do not lie on any hyperplane.
If $L \cap v_{d-s}(\BP W^{*}) \neq \emptyset$,
say $[\alpha^{d-s}] \in L$,
then the linear form $\alpha$ vanishes at each $[\eta_i]$,
so the $[\eta_i]$ lie on the hyperplane defined by $\alpha$, a contradiction.
\end{proof}

\begin{proof}[Proof of Theorem \ref{sigmathm}]
Suppose $\phi = \eta_1^d + \cdots + \eta_r^d$.
Consider the linear series $L = \BP \{ p\in S^{d-s}W^* \mid p( \eta_i )=0, 1\leq i\leq r\}$
as in Lemma~\ref{Lveremptylem}.
Then $L$ is contained in $\BP \Rker \phi_{s,d-s}$
so
\[
  r \geq \codim L \geq \codim \BP \Rker \phi_{s,d-s} = \rank \phi_{s,d-s} .
\]

\begin{remark}\label{Oldboundproof}Note that taking $r = R(\phi)$ proves  
  equation~\eqref{someeqns}, {\it a priori} just dealing with rank, but in fact also for border rank by 
the definition of Zariski closure.
\end{remark}

Now since $\BP \Rker \phi_{s,d-s}$ is a projective space,
if $\dim L + \dim \Sigma_{s} \geq \dim \BP \Rker \phi_{s,d-s}$
we would have
$L \cap \big( v_{d-s}(\BP W^{*}) \cap \BP\Rker\phi_{s,d-s} \big) \neq \emptyset$.
But by Lemma \ref{Lveremptylem} this intersection is empty.
Therefore
\[
  \dim L + \dim \Sigma_{s} < \dim \BP \Rker \phi_{s,d-s} .
\]
Taking codimensions in $\BP S^{d-s} W^{*}$,
we may rewrite this as
\[
  \codim L - \dim \Sigma_{s} > \codim \BP \Rker \phi_{s,d-s} = \rank \phi_{s,d-s} .
\]
Taking $r = R(\phi)$ yields
$R(\phi) \geq \codim L > \rank \phi_{s,d-s} + \dim \Sigma_{s}$.
\end{proof}

\begin{remark}\label{fermat is smooth}
If $\phi \in S^d W$ with $\langle \phi \rangle = W$ and $R(\phi) = n = \dim W$,
then the above theorem implies $\Sigma_1 = \emptyset$.
Note that this is easy to see directly: Writing $\phi = \eta_1^d + \dots + \eta_n^d$,
we must have $\langle \eta_1, \dots, \eta_n \rangle = \langle \phi \rangle = W$,
so in fact the $\eta_i$ are a basis for $W$.
Then the singular set of $\Zeros(\phi)$ is the common zero locus of the derivatives
$\eta_i^{d-1}$ in $\BP W$, which is empty.
\end{remark}

\begin{remark} The assumption that $\langle \phi \rangle = W$ is equivalent to $\Lker \phi_{1,d-1} = \{0\}$, i.e., that
  $\Zeros(\phi)$ is not a cone over a variety in a lower-dimension subspace.
It would be interesting to have a geometric  characterization of
the condition $\Lker \phi_{k,d-k} = \{0\}$ for $k > 1$.
\end{remark}

\begin{corollary}\label{cor: rank of reducible poly}
Let $n = \dim W$
and $\phi \in S^d(W)$ with $\langle \phi \rangle = W$.
If $\phi$ is reducible, then $R(\phi) \geq 2n-2$.
If $\phi$ has a repeated factor, then $R(\phi) \geq 2n-1$.
\end{corollary}
\begin{proof}
We have $\rank \phi_{1,d-1} = \dim W = n$.
If $\phi = \chi \psi$ factors, then $\Sigma_1(\phi)$ includes
the intersection $\{ \chi = \psi = 0 \}$, which has codimension $2$ in $\BP W \cong \BP^{n-1}$.
Therefore $R(\phi) \geq n+n-3+1 = 2n-2$.

If $\phi$ has a repeated factor, say $\phi$ is divisible by $\psi^2$,
then $\Sigma_1$ includes the hypersurface $\{ \psi = 0 \}$, which has codimension $1$.
So $R(\phi) \geq n+n-2+1 = 2n-1$.
\end{proof}

In the following sections we apply Theorem \ref{sigmathm} to several classes of polynomials.
Before proceeding we note the following extension.
\begin{proposition}
Let
\[
  \Sigma_{h,s}(\phi) = \bigcup_{\beta_1,\dots,\beta_h \in W^* \setminus\{0\}} \Sigma_s(\partial^h \phi / \partial \beta_1 \cdots \partial \beta_h).
\]
If $\Lker \phi_{h+1,d-h-1} = \{ 0 \}$ then for each $s$,
\[
  R(\phi) \geq \rank \phi_{s,d-s} + \dim \Sigma_{h,s} + 1.
\]
\end{proposition}
Theorem~\ref{sigmathm} is the case $h=0$.

Note that for $0 \leq j \leq k \leq d$, if $\Lker \phi_{k,d-k}=0$ then $\Lker \phi_{j,d-j}=0$.
Also note that $\Sigma_{s+1}(\phi) \subseteq \Sigma_s(\partial\phi/\partial\beta)$ for every $\beta\neq0$.
\begin{proof}
Let $\phi = \eta_1^d + \dots + \eta_r^d$ and let $L \subset \BP S^{d-s}(W^*)$
be the set of hypersurfaces of degree $d-s$ containing each $[\eta_i]$.
As before, $L$ is a linear subspace contained in $\BP \Rker \phi_{s,d-s}$.
Suppose $\alpha, \beta_1,\dots,\beta_h \in W^* \setminus \{0\}$ are such that
$[\alpha^{d-s-h} \beta_1 \cdots \beta_h] \in L$.
Then $\alpha \beta_1 \cdots \beta_h \in \Lker \phi_{h+1,d-h-1} = \{0\}$,
a contradiction.
Thus $L$ is disjoint from the set of points of the form $[\alpha^{d-s-h} \beta_1 \cdots \beta_h]$.

Now, $\alpha^{d-s-h} \beta_1 \cdots \beta_h \in \Rker \phi_{s,d-s}$ if and only if
$\alpha^{d-s-h} \in \Rker (\partial^h \phi / \partial \beta_1 \cdots \partial \beta_h)_{s,d-s-h}$,
and by Proposition~\ref{prop: kernel meeting veronese} this is equivalent to
$[\alpha] \in \Sigma_s(\partial^h \phi / \partial \beta_1 \cdots \partial \beta_h)$.
Therefore
\[
  \big\{ [\alpha^{d-s-h} \beta_1 \cdots \beta_h] \, \mid \, \forall \alpha, \beta_1, \dots, \beta_h \in W^* \setminus \{0\} \big\}
      \cap \BP \Rker \phi_{s,d-s} \cong \Sigma_{h,s} .
\]


We saw above that $L$ is disjoint from the left hand side.
Counting dimensions in $\BP \Rker \phi_{s,d-s}$, we get
\[
  \dim L + \dim \Sigma_{h,s} < \dim \BP \Rker \phi_{s,d-s} .
\]
Taking codimensions in $\BP S^{d-s}W^*$ yields the inequality
\[
  r \geq \dim L > \codim \BP \Rker \phi_{s,d-s} + \dim \Sigma_{h,s},
\]
where $\rank \phi_{s,d-s} = \codim \BP \Rker \phi_{s,d-s}$.
\end{proof}

One step in the proof above generalizes slightly:
With $L$ as in the proof,
if $[D] \in \BP L$ and $D$ factors as $D = \alpha_1^{a_1} \cdots \alpha_k^{a_k}$,
then $ \alpha_1 \cdots \alpha_k \in \Lker \phi_{k,d-k}$.
This idea already appeared in the proof of Theorem~\ref{sigmathm}
in the case $D = \alpha^s$.

\section{Ranks and border ranks of some cubic polynomials}\label{cubicsect}

\begin{proposition}\label{xyzprop}
Consider $\phi=x_1y_1z_1+\cdots + x_my_mz_m\in S^3W$, where $W=\BC^{3m}$.
Then $R(\phi)=4m=\frac 43\dim W$ and $\ur(\phi)=3m=\dim W$.
\end{proposition}
\begin{proof}
We have $\langle \phi \rangle = W$, so $\rank \phi_{1,2} = \dim W = 3m$,
and $\Sigma_1$ contains the set $\{ x_1=y_1=x_2=y_2=\dots=x_m=y_m=0\}$.
Thus $\Sigma_1$ has dimension at least $m-1$.
So $R(\phi) \geq 4m$ by Proposition~\ref{sigmathm}.
On the other hand, each $x_iy_iz_i$ has rank $4$ by Theorem~\ref{cubicranks}, so $R(\phi) \leq 4m$.
 
Since $\ur(xyz)=3$, we have $\ur(\phi)\leq 3m$.
On the other hand, one simply computes the matrix of $\phi_{1,2}$
and observes that it is a block matrix with rank at least $3m$.
Therefore $\ur(\phi) = 3m$.
\end{proof}

\begin{proposition}\label{lqprop}\ 
Let $\BC^{m+1}$ with $m>1$ have linear coordinates $x,y_1\hd y_m$.
Then,
\begin{enumerate}
\item $R(x (y_1^2+\cdots +y_m^2)) = 2m$.
\item $R(x (y_1^2+\cdots +y_m^2)+x^3)  = 2m$.
\end{enumerate}
\end{proposition}
\begin{proof}
Write $\phi = x (y_1^2+\cdots +y_m^2) \in S^3 W=S^3\BC^{m+1}$.
Then $R(\phi) \geq 2m$ by Corollary~\ref{cor: rank of reducible poly}.

Let $a_1, \dots, a_m$ be   nonzero complex numbers with $\sum a_i = 0$.
Write
\[
\begin{split}
  \phi &= x y_1^2 + \dots + x y_m^2 \\
         &= (x y_1^2 - a_1 x^3) + \dots + (x y_m^2 - a_m x^3) \\
         &= x(y_1 + a_1^{1/2} x)(y_1 - a_1^{1/2} x) + \dots + x(y_m + a_m^{1/2} x)(y_m - a_m^{1/2} x) .
\end{split}
\]
Each $x(y_j - a_j^{1/2} x)(y_j + a_j^{1/2} x)$ has rank $2$ by Theorem~\ref{csmainthm}.
Thus $\phi$ is the sum of $m$ terms which each have rank $2$, so $R(\phi) \leq 2m$.

The second statement follows by the same argument (with $\sum a_i = -1$).
\end{proof}

We have the bounds
\[
  m+1 = \rank \phi_{1,2} \leq \ur(\phi) \leq R(\phi) = 2m .
\]
It would be interesting to know the border rank of $x (y_1^2+\cdots +y_m^2)$
and $x(y_1^2+\cdots +y_m^2)+x^3$.

\begin{remark}
In particular, $x(y_1^2+y_2^2+y_3^2)$ has rank exactly $6$, which is strictly greater
than the generic rank $5$ of cubic forms in four variables.
(See Prop.~6.3 of \cite{Geramita} and the remark following it.)
\end{remark}

More generally,

\begin{proposition}
Let $\phi = x^2u + y^2v + xyz \in S^3 W$, $\dim W=5$. Then
$\ur(\phi)=5$ and
  $8\leq R(\phi) \leq 9$.
\end{proposition}
\begin{proof} The upper bound follows from the expression
\[
\begin{split}
  \phi = (&x+y+2^{1/3}z)^3 - (2^{2/3}x+z)^3 - (2^{2/3}y+z)^3 \\
    &- x^2 ( - u - 3x + 3y - 3 \cdot 2^{1/3} z )
    - y^2 ( -v + 3x - 3y - 3 \cdot 2^{1/3} z ) ,
\end{split}
\]
where the last two terms have the form $a^2 b$; recall that $R(a^2 b) = 3$.

To obtain the lower bound, note that 
the map $\phi_{1,2}$ is surjective, so $\codim \Rker \phi_{1,2} = \dim W = 5$.
In particular, $\ur(\phi) \geq 5$.
The singular set $\Sigma_1 = \{ x=y=0 \} \cong \pp{2}$.
Therefore $R(\phi) \geq 5+2+1 = 8$.

The upper bound for border rank follows by techniques explained in \S\ref{verlimits}.
Explicitly, define $5$ curves in $W$ as follows:
\[
  a(t) = x+t(u-z), \quad b(t) = y+t(v-z), \quad c(t) = (x+y)+tz, \quad d(t) = x+2y, \quad e(t) = x+3y ,
\]
and for $t\neq0$ let $P(t) \in \BP \Sym^3 W$ be   $P(t)=[a(t)^3+\cdots + e(t)^3]$,
so $P(t) \in \s_5(v_3(\BP W))$.
Note that $\tlim_{t\ra 0}P(t)$ is well defined.
A straightforward calculation as in \S\ref{verlimits}, \S\ref{monomialsect} shows that (after scaling coordinates)
$[\phi]=\lim_{t \to 0} P(t)$.
\end{proof}

\section{Plane cubic curves}\label{cubiccurvesect}

Throughout this section $\dim W = 3$.

Normal forms for plane cubic curves were determined in \cite{MR1506892} in the 1930's. In \cite{comonmour96} an explicit
algorithm was given for determining the rank of a cubic curve (building on unpublished work of Reznick), and  
the possible ranks for polynomials in each $\s_r(v_3(\pp 2))\backslash\s_{r-1}(v_3(\pp 2))$ were determined. Here we give the explicit
list of normal forms and their ranks and border ranks,
illustrating how one can use singularities of auxiliary geometric objects
to determine the rank of a polynomial.

\begin{theorem}\label{cubicranks}   The possible ranks and border ranks of plane cubic curves are described in  
Table~\ref{table: plane cubic curves}.
\end{theorem}

\begin{table}
\begin{tabular}{l l c r l}
\hline
Description & normal form & $R$ & $\ur$ & $\BP \Rker \phi_{1,2}\cap \s_2(v_2(\BP W^*))$  \\
\hline
triple line & $x^3$ & $1$ & $1$ & \\  
\hline
three concurrent lines & $xy(x+y)$ & $2$ & $2$ & \\  
\hline
double line + line & $x^2y$ & $3$ & $2$ &  \\  
\hline
irreducible & $y^2z - x^3 -  z^3$ & $3$ & $3$ & triangle\\
\hline
irreducible & $y^2z - x^3 - xz^2$ & $4$ & $4$ & smooth\\
\hline
cusp & $y^2z - x^3$ & $4$ & $3$ & double line + line \\
\hline
triangle & $xyz$ & $4$ & $4$ & triangle \\
\hline
conic + transversal line & $x(x^2+yz)$ & $4$ &$4$  & conic + transversal line \\
\hline
irreducible, smooth ($a^3 \neq -27/4$) & $y^2z - x^3 - axz^2 - z^3$ & $4$ &$4$  & irreducible, smooth cubic \\
\hline
irreducible, singular ($a^3 = -27/4$) & $y^2z - x^3 - axz^2 - z^3$ & $4$ &$4$  & irreducible, singular cubic \\
\hline
conic + tangent line & $y(x^2 + yz)$ & $5$ & $3$ & triple line \\
\hline
\end{tabular}
\caption{Ranks and border ranks of plane cubic curves.}\label{table: plane cubic curves}
\end{table}

The proof of Theorem \ref{cubicranks} given by \cite{comonmour96} relies first on
a computation of equations for the secant varieties $\s_k(v_3 (\BP W))$ for $2 \leq k \leq 3$,
which determines all the border ranks in Table \ref{table: plane cubic curves}.
Note that $\s_3(v_3 (\BP W))$ is a hypersurface defined by the Aronhold invariant, not
a symmetric flattening.
To refine the results to give the ranks of a non-generic point $\phi$ in each secant variety,
first \cite{comonmour96} uses the geometry of the {\it Hessian} of $\phi$
to distinguish some cases.
(The Hessian is the variety whose equation is the determinant
of the Hessian matrix of the equation of $\phi$.  Given a vector $v\in W^*$,
$\phi_{1,2}(v)$ in bases is the Hessian of $\phi$ evaluated at $v$.
When the curve $\Zeros(\phi)\subset \BP W^*$ is not a cone,
the variety $\BP \Rker \phi_{1,2}\cap \s_2(v_2(\BP W^*)) $ is
the Hessian cubic of $\phi$.)

The last case, $\phi = y(x^2+yz)$, is distinguished by an unpublished argument due to
B. Reznick.
Reznick shows by direct calculation that for any linear form $L$, the
geometry of the Hessian of $\phi-L^3$ implies $\phi-L^3$ has rank strictly greater than $3$;
so $\phi$ itself has rank strictly greater than $4$.
We thank Reznick for sharing  the details of this argument with us.

We exploit this connection to prove Theorem \ref{cubicranks}
by examining the geometry of the Hessian using the machinery we have set up
to study $\BP \Rker \phi_{1,2}$. We begin by computing the ranks of each cubic form.
We show that $\phi = y(x^2+yz)$ has rank $5$ 
by directly studying the Hessian of $\phi$ itself (rather than the modification $\phi-L^3$ as was done by Reznick).

\begin{proof}
Upper bounds for the ranks listed in Table~\ref{table: plane cubic curves}
are given by simply displaying an expression involving the appropriate number of terms.
For example, to show $R(xyz) \leq 4$, observe that
\[
  xyz = \frac{1}{24} \Big( (x+y+z)^3 + (x-y-z)^3 - (x-y+z)^3 - (x+y-z)^3 \Big) .
\]
We present the remainder of these expressions
in Table~\ref{plane cubic curves as sums}.

Next we show lower bounds for the ranks listed in Table~\ref{table: plane cubic curves}.
The first three cases are covered by Theorem~\ref{csmainthm}.
For all the remaining   $\phi$ in the table,  $\phi \notin \Sub_{2}$, so by  \eqref{subprop}, $R(\phi) \geq 3$.
By Remark~\ref{fermat is smooth}, if $\phi$ is singular then $R(\phi) \geq 4$,
and this is the case for the triangle, the union of a conic and a line, and the cusp
(but we will show the rank of the conic plus a tangent line is $5$).
We have settled all but the following three cases:
\[
  y^2z-x^3-xz^2, \quad y^2z-x^3-axz^2-z^3, \quad y(x^2+yz) .
\]
If $\phi = \eta_1^3 + \eta_2^3 + \eta_3^3$ with $[\eta_i]$ linearly independent,
then the Hessian of $\phi$ is defined by $\eta_1 \eta_2 \eta_3 = 0$,
so it is a union of three nonconcurrent lines.
In particular, it has three distinct singular points.
But a short calculation verifies that the Hessian of $y^2z-x^3-xz^2$ is smooth
and the Hessian of $y^2z-x^3-axz^2-z^3$ has at most one singularity.
Therefore these two curves have rank at least $4$, which agrees with the upper bounds
given in Table~\ref{plane cubic curves as sums}.

Let $\phi = y(x^2+yz)$.
The Hessian of $\phi$ is defined by the equation $y^3=0$.
Therefore the Hessian $\BP \Rker \phi_{1,2} \cap \s_2(v_2(\BP W))$ is a (triple) line.
Since it is not a triangle, $R(y(x^2+yz)) \geq 4$, as we have argued in the last two cases.
But in this case we can say more.

Suppose $\phi = y(x^2+yz) = \eta_1^3 + \eta_2^3 + \eta_3^3 + \eta_4^3$,
with the $[\eta_i]$ distinct points in $\BP W$.
Since $\langle \phi \rangle = W$, the $[\eta_i]$ are not all collinear.
Therefore there is  
a unique $2$-dimensional linear space of quadratic forms vanishing at the $\eta_i$.
These quadratic forms thus lie in $\Rker \phi_{1,2}$.
In the plane $\BP \Rker \phi_{1,2} \cong \pp{2}$,
$H := \BP \Rker \phi_{1,2} \cap \s_2(v_2(\BP W))$ is a triple line
and the pencil of quadratic forms vanishing at each $\eta_i$ is also a line  $L$.

Now either $H = L$ or $H \neq L$.
If $H=L$, then $L$ contains the point $\BP \Rker \phi_{1,2} \cap v_2(\BP W) \cong \Sigma_1$.
But $\langle \phi \rangle = W$, so $L$ is disjoint from $v_2(\BP W)$.
Therefore $H \neq L$.
But then $L$  contains exactly one reducible conic,
corresponding to the point $H \cap L$.
But this is impossible: a pencil of conics through four points in $\pp{2}$
contains at least three reducible conics (namely the pairs of lines through pairs of points).

Thus $\phi = y(x^2+yz) = \eta_1^3 + \eta_2^3 + \eta_3^3 + \eta_4^3$
is impossible, so $R(y(x^2+yz)) \geq 5$.

In conclusion, we have obtained for each cubic curve $\phi$ listed in Table \ref{table: plane cubic curves}
a lower bound $R(\phi) \geq m$ which agrees with the upper bound $R(\phi) \leq m$
as shown in Table \ref{plane cubic curves as sums}.
This completes the proof of the calculation of ranks.

Finally one may either refer  to the well-known characterization of degenerations of cubic curves
to find the border ranks; see for example \cite{MR1506892} or simply evaluate the defining equations
of the various secant varieties on the normal forms.
\end{proof}

 \begin{table} \tiny
\begin{align*}
xy(x+y) &= \frac{1}{3 \sqrt{3} i} \Big( (\omega x - y)^3 - (\omega^2 x - y)^3 \Big) \qquad (\omega = e^{2\pi i/3}) \\
x^2y &= \frac{1}{6} \Big( (x+y)^3 - (x-y)^3 - 2y^3 \Big) \\
y^2z - x^3 &= \frac{1}{6} \Big( (y+z)^3 + (y-z)^3 - 2z^3 \Big) - x^3 \\
xyz &= \frac{1}{24}( (x+y+z)^3 - (-x+y+z)^3 - (x-y+z)^3 - (x+y-z)^3 ) \\
x(x^2+yz) &= \frac{1}{96} \Big( (4x+y+z)^4 + (4x-y-z)^3 - 2(2x+y-z)^3 - 2(2x-y+z)^3 \Big) \\
y^2z - x^3 - xz^2 &= \frac{-1}{12\sqrt{3}} \Big( (3^{1/2}x + 3^{1/4}iy + z)^3 + (3^{1/2}x - 3^{1/4}iy + z)^3 \\
    & \qquad + (3^{1/2}x + 3^{1/4}y - z)^3 + (3^{1/2}x - 3^{1/4}y - z)^3  \Big) \\
y^2z - x^3 - z^3 &= \frac{1}{6 \sqrt{3} i} \Big( ((2\omega + 1) z - y)^3 - ((2\omega^2 + 1)z - y)^3 \Big) - x^3 \\
y^2z - x^3 - axz^2 - z^3 &= z(y-z)(y+z) - x(x-a^{1/2}iz)(x+a^{1/2}iz) \\
    &= \frac{1}{6\sqrt{3}i} \Big( (2\omega z - (y-z))^3 - (2\omega^2 z - (y-z) )^3 \Big) \\
    & \quad - \frac{1}{6\sqrt{3}i} \Big( ( \omega(x-a^{1/2}iz) - (x+a^{1/2}iz) )^3 - (\omega^2(x-a^{1/2}iz) - (x+a^{1/2}iz))^3 \Big) \\
y(x^2 + yz) &= (x-y)(x+y)y + y^2(y+z) \\
    &= \frac{1}{6\sqrt{3}i} \Big( (2\omega y - (x-y))^3 - (2\omega^2 y - (x-y))^3 \Big)
        + \frac{1}{6} \Big( (2y+z)^3 + z^3 - 2(y+z)^3 \Big)
\end{align*}
\caption{Upper bounds on ranks of plane cubic forms.}\label{plane cubic curves as sums}
\end{table}

\section{Determinants and permanents}\label{detpersect}

 Let $X$ be an $n \times n$ matrix whose entries $x_{i,j}$ are variables forming a basis for $W$.
Let $\det_n = \det X$ and $\text{per}_n$ be the permanent of $X$.

In \cite{Gurvits}, L. Gurvits applied the equations for flattenings  \eqref{someeqns} to the determinant and permanent polynomials
to observe, for each $1 \leq a \leq n-1$,
\[
  \rank (\det\nolimits_n)_{a,n-a} = \rank (\text{per}_n)_{a,n-a} = \binom{n}{a}^2 ,
\]
giving lower bounds for border rank.
(In \cite{Gurvits} he is only concerned with rank
but he only uses \eqref{someeqns} for lower bounds.)
Indeed, the image of $(\det\nolimits_n)_{a,n-a}$ is spanned by the determinants of $a \times a$ submatrices
of $X$, and the image of $(\text{per}_n)_{a,n-a}$ is spanned by the permanents of $a \times a$ submatrices
of $X$.
These are independent and number $\binom{n}{a}^2$.
In the same paper Gurvits also gives upper bounds as follows.
\[
  R(\det\nolimits_n) \leq 2^{n-1} n!, \qquad R(\text{per}_n) \leq 4^{n-1} .
\]
The first bound follows by writing $\det\nolimits_n$ as a sum of $n!$ terms, each of
the form $x_1 \cdots x_n$, and applying Proposition~\ref{hplaneprodprop}: $R(x_1 \cdots x_n) \leq 2^{n-1}$.
For the second bound, a variant of the Ryser formula for the permanent (see \cite{MR0150048}) allows one to
  write $\text{per}_n$ as a sum of $2^{n-1}$ terms, each of the form $x_1 \cdots x_n$:
\[
  \text{per}_n = 2^{-n+1} \sum_{\substack{\epsilon \in \{-1,1\}^n \\ \epsilon_1=1}}
    \prod_{1 \leq i \leq n} \sum_{1 \leq j \leq n} \epsilon_i \epsilon_j x_{i,j} ,
\]
the outer sum taken over $n$-tuples $(\epsilon_1=1, \epsilon_2,\dots,\epsilon_n)$.
Note that each term in the outer sum is a product of $n$ independent linear forms
and there are $2^{n-1}$ terms.
Applying Proposition~\ref{hplaneprodprop} again gives the upper bound for $R(\text{per}_n)$.

Now, we apply Theorem~\ref{sigmathm} to improve the lower bounds for rank.
The determinant $\det_n$ vanishes to order $a+1$ on a matrix $A$ if and only if
every minor of $A$ of size $n-a$ vanishes.
Thus $\Sigma_a(\det_n)$ is the locus of matrices of rank at most $n-a-1$.
This locus has   dimension $n^2-1-(a+1)^2$.
Therefore, for each $a$,
\[
  R(\det\nolimits_n) \geq \binom{n}{a}^2 + n^2 - (a+1)^2 .
\]
The right hand side is maximized at $a = \lfloor n/2 \rfloor$.

A crude lower bound for $\dim \Sigma_a(\text{per}_n)$ is obtained as follows.
If a matrix $A$ has $a+1$ columns identically zero, then each term in $\text{per}_n$ vanishes to order $a+1$,
so $\text{per}_n$ vanishes to order at least $a+1$.
Therefore $\Sigma_a(\text{per}_n)$ contains the set of matrices with $a+1$ zero columns,
which is a finite union of projective linear spaces of dimension $n(n-a-1)-1$.
Therefore, for each $a$,
\[
  R(\text{per}_n) \geq \binom{n}{a}^2 + n(n-a-1) .
\]
Again, the right hand side is maximized at $a = \lfloor n/2 \rfloor$.

See Table~\ref{table: det perm} for values of the upper bound for rank
and lower bound for border rank obtained by Gurvits and the lower bound
for rank given here.

\begin{table}
\begin{tabular}{l | rrrr rrr}
$n$ & $2$ & $3$ & $4$ & $5$ & $6$ & $7$ & $8$ \\
\hline
Upper bound for $R(\det_n)$ & $4$ & $24$ & $192$ & $1{,}920$ & $23{,}040$ & $322{,}560$ & $5{,}160{,}960$ \\
  Lower bound   for $R(\det_n)$                    & $4$ & $14$ & $43$ & $116$ & $420$ & $1{,}258$ & $4{,}939$ \\
  Lower bound for $\ur(\det_n)$     & $4$ & $9$ & $36$ & $100$ & $400$ & $1{,}225$ & $4{,}900$ \\
\hline
Upper bound for $R(\text{per}_n)$ & $4$ & $16$ & $64$ & $256$ & $1{,}024$ & $4{,}096$ & $16{,}384$ \\
  Lower bound   for $R(\text{per}_n)$                  & $4$ & $12$ & $40$ & $110$ & $412$ & $1{,}246$ & $4{,}924$ \\
  Lower bound  for $\ur(\text{per}_n)$    & $4$ & $9$ & $36$ & $100$ & $400$ & $1{,}225$ & $4{,}900$ \\
\end{tabular}
\caption{Bounds for determinants and permanents.}\label{table: det perm}
\end{table}

\section{Limits of secant planes for Veronese varieties}\label{verlimits} 
 
\subsection{Limits of secant planes for arbitrary projective varieties}\label{limitsubsect} 

Let $X\subset \BP V$ be a projective variety.
Recall that $\s_r^0(X)$ denotes the set of   points on $\s_r(X)$
that lie on a $\pp{r-1}$ spanned by $r$ points on $X$.
We work inductively, so we assume we know
the nature of points on $\s_{r-1}(X)$
and study points on $\s_r(X)\backslash (\s_r^0(X)\cup \s_{r-1}(X))$.

It is convenient to study the limiting $r$-planes as
  points on the     Grassmannian 
in its Pl\"ucker embedding, $G(r,V)\subset \BP(\bigwedge^r V)$.
I.e., we consider the curve of $r$ planes as being
represented by $[x_1(t)\ww \cdots\ww x_r(t)]$, where $x_j(t)\subset \hat X\backslash 0$ and
examine the limiting plane as $t\to 0$. (There is
  a unique such plane as the Grassmannian is compact.)
 
Let $[p]\in \s_r(X)$. Then
there exist curves $x_1(t)\hd x_r(t)\subset \hat X$ with
$p\in \lim_{t\to 0}\langle x_1(t)\hd x_r(t)\rangle $.
We are interested in the case when
$\dim \langle x_1(0)\hd x_r(0)\rangle <r$. 
(Here $\langle v_1\hd v_k\rangle$ denotes  the linear span
of the vectors $v_1\hd v_k$.) 
Use the notation
$x_j=x_j(0)$.    Assume 
for the moment 
that $x_1\hd x_{r-1}$ are linearly independent.
Then we may write $x_r=c_1x_1+\cdots + c_{r-1}x_{r-1}$ for some
constants $c_1\hd c_{r-1}$. Write each curve
$x_j(t)=x_j+tx_j'+\frac 12 t^2x_j'' +\cdots$ where derivatives are taken at $t=0$.

Consider the Taylor series  
\begin{align*}
&x_1(t)\ww\cdots\ww x_r(t) =\\ 
&(x_1+tx_1'+\frac 12 t^2x_1'' +\cdots)\ww \cdots \ww (x_{r-1}+tx_{r-1}'+\frac 12 t^2x_{r-1}'' +\cdots)
\ww (x_{r}+tx_r'+\frac 12 t^2x_r'' +\cdots) 
\\
&=t((-1)^r(c_1x_1' +\cdots c_{r-1}x_{r-1}'-x_r')\ww x_1\ww\cdots\ww x_{r-1}) + t^2(...) +\cdots
\end{align*}

If the $t$ coefficient is nonzero, then $p$
lies in the the $r$ plane   $\langle x_1\hd x_{r-1},(c_1x_1' +\cdots c_{r-1}x_{r-1}'-x_r')\rangle$.

If the $t$ coefficient is zero, then $c_1x_1'+\cdots + c_{r-1}x_{r-1}'-x_r'=
e_1x_1+\cdots e_{r-1}x_{r-1}$ for some constants $e_1\hd e_{r-1}$.
In this case we must examine the $t^2$ coefficient of the expansion. It is
\[ 
  \left( \sum_{k=1}^{r-1} e_kx_k' + \sum_{j=1}^{r-1}c_jx_j''-x_r'' \right) \ww x_1\ww\cdots\ww x_{r-1} .
\]
One continues to higher order terms if this is zero.

For example, when $r=3$, the $t^2$ term is
\begin{equation}\label{ttwoterm}
x_1'\ww x_2'\ww x_3+
x_1'\ww x_2 \ww x_3'+x_1 \ww x_2'\ww x_3'
+x_1''\ww x_2 \ww x_3+
x_1\ww x_2'' \ww x_3+
x_1\ww x_2 \ww x_3'' .
\end{equation}

\subsection{Limits for Veronese varieties}
For any smooth variety $X\subset \BP V$, a point on
$\s_2(X)$ is either a point of $X$, a point on an
honest secant line (i.e., a point of $X$-rank two) or
a point on a tangent line of $X$. For a Veronese variety
all nonzero tangent vectors are equivalent. They
are all of the form $x^d+x^{d-1}y$ (or equivalently $x^{d-1}z$),
in particular they lie on a subspace variety $\Sub_2$
and thus have rank $d$ by Theorem \ref{csmainthm}. 
In summary:

\begin{proposition}\label{ver2prop}
If $p\in \s_2(v_d(\BP W))$ then $R(p)=1$, $2$ or $d$.
In these cases $p$
respectively has the normal forms $x^d,x^d+y^d,x^{d-1}y$.
(The last two are equivalent when $d=2$.)
\end{proposition}

We consider the case of points on $\s_3(v_d(\BP W))\backslash \s_2(v_d(\BP W))$.  
We cannot have three distinct limiting points $x_1,x_2,x_3$ with $\dim\langle x_1,x_2,x_3\rangle<3$ unless at least two of them coincide because there are no trisecant lines to $v_d(\BP W)$.  
(For any variety  $X\subset \BP V$   with ideal generated in degree two,
  any trisecant line of $X$ is contained in $X$, and
  Veronese varieties $v_d(\BP W)\subset \BP S^dW$ are cut out by quadrics but contain no
lines.)

Write our curves as
\begin{align*}
x(t) &= (x_{0} + t x_{1}+ t^2 x_{2} + t^3 x_{3} + \cdots)^d \\
&= x_0^d + t (d x_0^{d-1} x_1) + t^2 \left( \binom{d}{2} x_0^{d-2} x_1^2 + d x_0^{d-1} x_2 \right) \\
&\quad  + t^3 \left( \binom{d}{3} x_0^{d-3} x_1^3 +   d(d-1) x_0^{d-2} x_1 x_2 + d x_0^{d-1} x_3 \right) + \cdots
\end{align*}
and similarly for $y(t)$, $z(t)$.

Case 1: two distinct limit points $x_0^d$, $z_0^d$, with $y_0=x_0$.
(We can always rescale to have equality of points rather than just collinearity
since we are working in projective space.) 
When we expand the
Taylor series, assuming $d>2$ (since the $d=2$ case is
well understood and different), the coefficient of $t$ (ignoring constants which disappear
when projectivizing) is
\[
   x_0^{d-1} ( x_1 - y_1 ) \ww x_0^d \ww z_0^d
\]
which can be zero only if   $x_1 \equiv y_1 \mod x_0$. If this 
examining \eqref{ttwoterm} we see the second order term is of the form
\[
  x_0^{d-1} ( x_2 - y_2 + \lambda x_1 )\ww x_0^d \ww z_0^d .
\]
Similarly if this term vanishes, the $t^3$ term will still be of the same nature.
Inductively, if the lowest nonzero term is $t^k$ then for each $j<k$, $y_j=x_j \mod (x_0,\dots,x_{j-1})$,
and the coefficient of the $t^k$ term is (up to a constant factor)
\[
  x_0^{d-1} ( x_k - y_k + \ell) \ww x_0^d \ww z_0^d
\]
where $\ell$ is a linear combination of $x_0, \dots, x_{k-1}$.
We rewrite this as $x^{d-1}y \ww x^d \ww z^d$.
If $\dim\langle z,x,y\rangle< 3$ we are reduced to a point of $\s_3(v_d(\pp 1))$ and
can appeal to Theorem \ref{csmainthm}. If the span is three
dimensional then any point in the plane $[x^{d-1}y \ww x^d \ww z^d]$
can be put in the normal form $x^{d-1}w+z^d$.

Case 2: One limit point $x_0=y_0=z_0=z$.
The $t$ coefficient vanishes and the $t^2$ coefficient is (up to a constant factor)
\[
  x_0^{d-1} (x_1-y_1) \ww x_0^{d-1} (y_1-z_1) \ww x_0^d
\]
which can be rewritten as $x^{d-1}y \ww x^{d-1}z \ww x^d$.
If this expression is nonzero then any point in the plane $[x^{d-1}y \ww x^{d-1}z \ww x^d]$
lies in $\s_2(v_d(\pp 1))$.
So we thus assume the $t^2$ coefficent vanishes.
Then $y_1-z_1$, $x_1-y_1$, and $x_0$ are linearly dependent;
a straightforward   calculation shows that the $t^3$ coefficient is
\[
  x_0^d \ww x_0^{d-1} (y_1-x_1) \ww ( x_0^{d-1} \ell + (\lambda^2+\lambda)(y_1-x_1)^2 ) ,
\]
where $\ell$ is a linear combination of $x_0, \dots, z_2$.
We rewrite this as $x^d \ww x^{d-1} y \ww (x^{d-1} \ell + \mu x^{d-2} y^2)$.
If $\mu=0$, every point in the plane $[x^d \ww x^{d-1}y \ww x^{d-1}\ell]$
lies in $\s_2(v_d(\pp 1))$, so we apply Theorem~\ref{csmainthm}.
If $\mu\neq0$ and $x^d \ww x^{d-1}y \ww (x^{d-1}\ell+\mu x^{d-2}y^2)=0$,
then $x, y$ are linearly dependent; then one considers higher powers of $t$.
One can argue that the lowest nonzero term
always has the form $x^d \ww x^{d-1}y \ww (x^{d-1}\ell+\mu x^{d-2}y^2)$. 

Thus our point lies in a plane of the form $[x^d \ww x^{d-1}y \ww (x^{d-1}\ell+\mu x^{d-2}y^2)]$.

\begin{theorem}\label{normalformsprop}
There are three  types of points $\phi\in S^3W$ of border rank three
with $\dim \langle \phi \rangle = 3$.
They have the following normal forms:
\[
\begin{array}{|c|c|c|}
\hline
{\rm limiting \ curves \ }& {\rm normal\ form}&  R \\
x^d,y^d,z^d                           & x^d+y^d+z^d         & 3\\
x^d,(x+ty)^{d},z^d                   & x^{d-1}y+z^d         & d \leq R \leq d+1\\
x^d,(x+ty)^{d},(x+2ty+t^2z)^{d} &   x^{d-2}y^2+x^{d-1}z & d \leq R\leq 2d-1\\
\hline
\end{array}
\] 
\end{theorem}

The upper bounds on ranks come from computing the sum of the ranks of the terms.
The lower bounds on ranks are attained by specialization to $S^d\BC^2$. 
We remark that when $d=3$, the upper bounds on rank are attained in both cases.

\begin{corollary} Let $\phi\in S^dW$ with $\ur(\phi)=3$. If
$R(\phi)>3$, then $2d-1\geq R(\phi)\geq d-1$ and only three values
occur, one of which is $d-1$.
\end{corollary}

\begin{proof} The only additional cases occur if $\dim \langle \phi\rangle=2$ which
are handled by Theorem \ref{csmainthm}.
\end{proof}

Even for higher secant varieties,  $x_1^d\wedge \cdots \wedge x_r^d $ cannot be zero if the $x_j$ are distinct
points, even if they lie on a $\pp 1$,  as long as $d\geq r$. This is because
a hyperplane in $S^dW$ corresponds to  
a (defined up to scale) homogeneous polynomial of degree $d$ on $W$. Now take $W=\BC^2$.
No homogeneous
polynomial of degree $d$ vanishes at $d+1$ distinct points of $\pp 1$, thus the image of any $d+1$
distinct points under the $d$-th Veronese embedding cannot lie on a hyperplane. 
As long as the degree is sufficiently large,  there is  a similar phenomenon
for tangent lines and higher osculating spaces (e.g.,  the intersection of the
the embedded tangent space to the Veronese with the Veronese is the point
of tangency when $d>2$, the intersection of the second osculating space of
a point with the Veronese is just that point if $d>3$ etc...). Because of
this,  when taking limits of small numbers of points (small with respect to $d$),
all limits are sums of limits to distinct points as in \S\ref{verallcrash} below.
These remarks prove Theorems \ref{normalformsprop4} 
and \ref{normalformsprop5} below.

\begin{theorem} \label{normalformsprop4} There are six  types of points of border
rank four in $S^dW$, $d>2$, whose span is $4$ dimensional. They
  have the following normal forms:
\[
\begin{array}{| l |c|c|c|}
\hline
{\rm limiting \ curves \ }& {\rm normal\ form}&  R  \\
x^d,y^d,z^d,w^d & x^d+y^d+z^d+w^d & 4 \\
x^d,(x+ty)^d,z^d,w^d& x^{d-1}y+z^d+w^d & d \leq R\leq d+2 \\
x^d,(x+ty)^d ,z^d,(z+tw)^d & x^{d-1}y+z^{d-1}w& d \leq R\leq 2d \phantom{+0} \\
x^d,(x+ty)^d,(x+ty+t^2z)^d,(x+t^2z)^d  & x^{d-2}yz  & d \leq R\leq 2d-2 \\
x^d,(x+ty)^d,(x+ty+t^2z)^d,w^d &   x^{d-2}y^2+x^{d-1}z+w^d & d \leq R\leq 2d \phantom{+0} \\
x^d,(x+ty)^d,(x+ty+t^2z)^d, &   x^{d-3}y^3+x^{d-2}z^2+x^{d-1}w & d \leq R \leq 3d-3 \\
\quad (x+ty+t^2z+t^3w)^d & & \\
\hline
\end{array}
\] 
\end{theorem}

\medskip
 
For $\s_5(v_d(\BP W))$, we get a new phenomenon when $d=3$ because $\dim S^3\BC^2=4<5$.
We can have $5$ curves $a, b, c, d, e$, with $a_0, \dots, e_0$ all lying in a $\BC^2$, but otherwise general,
so $\dim \langle a_0^3 , \dots , e_0^3 \rangle = 4$.
Thus the $t$ term will be of the form
$a_0^3 \ww b_0^3 \ww c_0^3 \ww d_0^3 \ww (s_1 a_0^2 a_1 + \cdots + s_4 d_0^2 d_1 - e_0^2 e_1)$.
Up to scaling we can give $\BC^2$ linear coordinates $x, y$ so that $a_0=x$, $b_0=y$,
$c_0=x+y$, $d_0=x+\lambda y$ for some $\lambda$.
Then, independent of $e_0$, the limiting plane will be contained in
\[
  \langle x^3 , y^3 , (x+y)^3 , (x+\lambda y)^3 , x^2 \alpha, xy \beta, y^2 \gamma \rangle
\]
for some $\alpha, \beta, \gamma \in W$ (depending on $a_1,\dots,e_1$).
Any point contained in the plane is of the form
$x^2 u + y^2 v + xyz$ for some $u, v, z \in W$.

\begin{theorem} \label{normalformsprop5} 
 There are seven types of points   of border rank five in $S^dW$ whose span is five dimensional when $d>3$,
and eight types when $d=3$. Six of the types are obtained by adding a term of the form
$u^d$ to a point of border rank four, the seventh has the normal form
$x^{d-4}u+ x^{d-3}y^3+x^{d-2}z^2+x^{d-1}w$, and the eighth type, which occurs
when $d=3$, has normal form  $x^2u+y^2v+xyz$.
\end{theorem}

\begin{remark} By dimension count, we expect to have normal forms
of elements of $\s_r(v_d(\pp {n-1}))$ as long as $r\leq n$ because
$\dim \s_r(v_d(\pp {n-1}))\leq rn-1$ and $\dim GL_n=n^2$.
\end{remark}

\section{Monomials}\label{monomialsect}

\subsection{Limits of highest possible osculation}\label{verallcrash}
Let $x(t)\subset W$ be a curve, write $x_0=x(0)$, $x_1=x'(0)$ and $x_j=x^{(j)}(0)$. Consider
the corresponding curve $y(t)=x(t)^d$ in $\hat v_d(\BP W)$ and note that
\begin{align*}
y(0)&=x_0^d\\
y'(0)&=dx_0^{d-1}x_1\\
y''(0)&=d(d-1)x_0^{d-2}x_1^2+dx_0^{d-1}x_2\\
y^{(3)}(0)&=d(d-1)(d-2)x_0^{d-3}x_1^3+3d(d-1)x_0^{d-2}x_1x_2+dx_0^{d-1}x_3\\
y^{(4)}(0)&=d(d-1)(d-2)(d-3)x_0^{d-4}x_1^4+6d(d-1)(d-2)x_0^{d-3}x_1^2x_2+3d(d-1)x_0^{d-2}x_2^2 \\
&\quad + 4d(d-1)x_0^{d-2}x_1x_3+dx_0^{d-1}x_4\\
y^{(5)}(0)&=d(d-1)(d-2)(d-3)(d-4)x_0^{d-5}x_1^5+9d(d-1)(d-2)(d-3)x_0^{d-4}x_1^3x_2 \\
&\quad + 10d(d-1)(d-2)x_0^{d-3}x_1^2x_3 +15d(d-1)(d-2)x_0^{d-3}x_1x_2^2 \\
&\quad + 4d(d-1)x_0^{d-2}x_2x_3+5d(d-1)x_0^{d-2}x_1x_4+dx_0^{d-1}x_5 \\
&\vdots
\end{align*}
At $r$ derivatives, we get a sum of terms
\[
x_0^{d-s}x_1^{a_1}\cdots x_p^{a_p}, \ \ a_1+2a_2+\cdots + pa_p=r,\ \ s=a_1+\cdots +a_p.
\]
In particular, $x_0x_1\cdots x_{d-1}$ appears for the first
time at the $1+2+\cdots +(d-1)= \binom d2$ derivative.

 \subsection{Bounds for monomials}

Write $\bbb=(b_1\hd b_m)$. Let $S_{\bbb,\d}$ denote the number of
distinct $m$-tuples $(a_1\hd a_m)$ satisfying $a_1+\cdots +a_m=\d$
and $0 \leq a_j \leq b_j$. Adopt the notation that $\binom ab=0$ if $b>a$ and is the usual binomial
coefficient otherwise.
We thank L. Matusevich for the following expression:
\begin{proposition}\label{monomiallowerbound}
Write $I=(i_1,i_2\hd i_k)$ with $i_1\leq i_2\leq \cdots \leq i_k$. Then
\[
  S_{\bbb,\d} = \sum_{k=0}^m (-1)^k \left[ \sum_{|I|=k} \binom{ \d+m-k-(b_{i_1}+\cdots+b_{i_k}) }{m } \right] .
\]
\end{proposition}
\begin{proof}
The proof is safely left to the reader.
(It is a straightforward inclusion-exclusion counting argument,
in which the $k$th term of the sum counts the $m$-tuples with $a_j \geq b_j+1$
for at least $k$ values of the index $j$.)
For those familiar with algebraic geometry, note
that $S_{\bbb,\d}$ is the Hilbert function in degree $\delta$ of the variety defined by the monomials
$x_1^{b_1+1}\hd x_m^{b_m+1}$.
\end{proof}

For $\bbb=(b_1\hd b_m)$, consider the quantity
\[
  T_{\bbb} := \prod_{i=1}^m (1 + b_i) .
\]
$T_{\bbb}$ counts the number of tuples $(a_1,\dots,a_m)$ satisfying $0 \leq a_j \leq b_j$
(with no restriction on $a_1 + \dots + a_m$).

\begin{theorem}\label{urmonomialest}  
Let $b_0\geq b_1\geq \cdots \geq b_n$ and   write $d = b_0 + \cdots + b_n$. Then
\[
  S_{(b_0,b_1\hd b_n),\lfloor \frac  d 2\rfloor}
    \leq  \ur(x_0^{b_0}x_1^{b_1}\cdots x_n^{b_n}) \leq  T_{(b_1\hd b_n)}.
\]
\end{theorem}
\begin{proof}
Let $\phi = x_0^{b_0} \cdots x_n^{b_n}$.
The lower bound follows from considering the image of
$\phi_{\lfloor \frac  d 2\rfloor,\lceil \frac  d 2\rceil}$, 
which is
\[
  \phi_{\lfloor \frac  d 2\rfloor,\lceil \frac  d 2\rceil}(S^{\lceil \frac  d 2\rceil} \BC^{n+1} )
  = \Big\langle x_0^{a_0}x_1^{a_1}\cdots x_n^{a_n}
  \, \Big| \, 0 \leq a_j \leq b_j , \,  a_0+a_1+\cdots+a_n =  \Big\lfloor \frac{d}{2} \Big\rfloor \Big\rangle
\]
whose dimension is $S_{(b_0,b_1\hd b_n),\lfloor \frac{d}{2} \rfloor}$.
 
We show the upper bound as follows.
Let $\bbb=(b_0,\dots,b_n)$ and
\begin{equation}\label{bigexpansion}
  F_{\bbb}(t) = \bigwedge_{s_1=0}^{b_1} \cdots \bigwedge_{s_n=0}^{b_n}
    (x_0 + t^1 \lambda_{1,s_1} x_1 + t^2 \lambda_{2,s_2} x_2 + \dots + t^n \lambda_{n,s_n} x_n)^d
\end{equation}
where the $\lambda_{i,s}$ are chosen sufficiently generally.
We may take each $\lambda_{i,0}=0$ and each $\lambda_{i,1}=1$ if we wish.
For $t \neq 0$, $[F_{\bbb}(t)]$ is a plane spanned by $T_{\bbb}$ points in $v_d(\BP W)$.
We claim $x_0^{b_0} \cdots x_n^{b_n}$ lies in the plane $\lim_{t \to 0} [F_{\bbb}(t)]$,
which shows $\ur(x_0^{b_0} \cdots x_n^{b_n}) \leq T_{\bbb}$.
In fact, we claim
\begin{equation}\label{eqn:bigwedgelimit}
  \lim_{t \to 0} [F_{\bbb}(t)] = \left[ \bigwedge_{a_1=0}^{b_1} \cdots \bigwedge_{a_n=0}^{b_n}
    x_0^{d-(a_1+\dots+a_n)} x_1^{a_1} \cdots x_n^{a_n} \right]
\end{equation}
so $x_0^{b_0} \cdots x_n^{b_n}$ occurs precisely as the last member of the spanning set
for the limit plane.

The coefficients of terms in $\lim_{t \to 0} F_{\bbb}(t)$ are given by determinants of certain matrices, as follows.
For an $n$-tuple $I=(a_1,\dots,a_n)$ and an $n$-tuple $(p_1,\dots,p_n)$ satisfying $0\leq p_i \leq b_i$,
let
\[
  c_{(p_1,\dots,p_n)}^{(a_1,\dots,a_n)} = \lambda_{1,p_1}^{a_1} \cdots \lambda_{n,p_n}^{a_n} ,
\]
the coefficient of $x_1^{a_1} \cdots x_n^{a_n} x_0^{d-(a_1+\cdots+a_n)}$
in $(x_0 + t \lambda_{1,p_1} x_1 + \cdots + t^n \lambda_{n,p_n}x_n)^d$,
omitting binomial coefficients.
Choose an enumeration of the $n$-tuples $(p_1,\dots,p_n)$ satisfying $0 \leq p_i \leq b_i$;
say, in lexicographic order.
Then given $n$-tuples $I_1,\dots,I_{T_\bbb}$, the coefficient of the term
\[
  x^{I_1} \ww \cdots \ww x^{I_{T_\bbb}}
\]
in $F_{\bbb}(t)$ is the product $\prod_{j=1}^{T_\bbb} c_j^{I_j}$, omitting binomial coefficients.
We may interchange the $x^{I_j}$ so that $I_1 \leq \dots \leq I_{T_\bbb}$ in some order, say lexicographic.
Then the total coefficient of $x^{I_1} \ww \cdots \ww x^{I_{T_\bbb}}$ is the alternating sum of the permuted products,
\[
  \sum_{\pi} (-1)^{|\pi|} \prod_{j=1}^{T_\bbb} c_{\pi(j)}^{I_j} ,
\]
(summing over all permutations $\pi$ of $\{1,\dots,T_\bbb\}$)
times a product of binomial coefficients (which we henceforth ignore).
This sum is the determinant of the $T_\bbb \times T_\bbb$ matrix $C := ( c_i^{I_j} )_{i,j}$.

First we show that for the term in \eqref{eqn:bigwedgelimit}, $\det C \neq 0$, i.e., the term does appear with
a non-zero coefficient in $\lim_{t \to 0} [F_{\bbb}(t)]$.
This is the term 
$x^{I_1} \wedge \cdots \wedge x^{ I_{T_\bbb}}$ where $I_1, \dots, I_{T_\bbb}$
is an enumeration of the set of tuples $\{(a_1,\dots,a_n) \mid 0 \leq a_i \leq b_i \}$.
For this term, $C$ is a tensor product,
\[
  C = (\lambda_{1,i}^j)_{i,j=0}^{b_1} \otimes \cdots \otimes (\lambda_{n,i}^j)_{i,j=0}^{b_n} .
\]
Since the matrices on the right are Vandermonde and the $\lambda_{k,i}$ are distinct, they are all nonsingular.
Therefore so is $C$.

Next we show that all other terms $x^{I_1} \wedge \cdots \wedge x^{I_{T_\bbb}}$ have coefficient $\det C=0$
or appear in $F_{\bbb}(t)$ with a strictly greater power of $t$ than the term in \eqref{eqn:bigwedgelimit} (or both),
so that the term in \eqref{eqn:bigwedgelimit} is the only term surviving in $\lim_{t \to 0} [F_\bbb(t)]$.

We may assume the monomials $x^{I_1}, \dots, x^{I_{T_\bbb}}$ are all distinct (otherwise the term
$x^{I_1} \ww \cdots \ww x^{I_{T_\bbb}}$ vanishes identically).

Let $r=x_2^{r_2} \cdots x_n^{r_n}$ and $p=d-\deg(r) \geq 0$.
We claim that if $x_0^{p-q} x_1^q r$ occurs among the $x^{I_j}$ for more than $b_1+1$ values of $q$,
then $\det C=0$.
Reordering the $I_j$ if necessary, say
\[
  x^{I_1} = x_0^{p-q_1} x_1^{q_1} r, \quad \dots \quad , \quad
  x^{I_{b_1+2}} = x_0^{p-q_{b_1+2}} x_1^{q_{b_1+2}} r .
\]
Let $C'$ be the first $b_1+2$ columns of $C$.
Then $C'$ is a tensor product:
\[
  C' = (\lambda_{1,i}^{q_j})_{\substack{i=0,\dots,b_1 \\ j=1,\dots,b_1+2}}
    \otimes (\lambda_{2,i}^{r_2})_{i=0}^{b_2} \otimes \cdots \otimes (\lambda_{n,i}^{r_n})_{i=0}^{b_n} .
\]
Here the first matrix has size $(b_1+1) \times (b_1+2)$ and the rest are column vectors, $(b_i+1) \times 1$.
The columns of the first matrix are dependent, hence so are the columns of $C'$, which are just columns of $C$.
This shows $\det C = 0$.

More generally, if $r$ is any monomial in $(n-1)$ of the variables, say $x_1, \dots, x_{i-1}, x_{i+1}, \dots, x_n$,
then $x_i^q r$ can occur for at most $b_i+1$ distinct values of the exponent $q$.
The lowest power of $t$ occurs when the values of $q$ are $q=0,1,\dots$.
In particular $x_i^q r$ only occurs for $q \leq b_i$.

Therefore, if a term $x^{I_1} \ww \cdots \ww x^{I_{T_\bbb}}$ has a nonzero coefficient in $F_{\bbb}(t)$
and occurs with the lowest possible power of $t$,
then in every single $x^{I_j}$, each $x_i$ occurs to a power $\leq b_i$.
The only way the $x^{I_j}$ can be distinct is for it to be the term in the right hand side of 
\eqref{eqn:bigwedgelimit}.
This shows that no other term with the same or lower power of $t$ survives in $F_{\bbb}(t)$.
\end{proof}

For example
\[
\begin{split}
  F_{(b)}(t) &= x_0^d\ww \bigwedge_{s=1}^b( x_0+t \l_{s}x_1)^d\\
& = t^{\binom {b+1}{2}} \left[ x_0^d\ww \left( \sum (-1)^s\l_s \right) x^{d-1}_0x_1\ww 
\sum (-1)^{s+1}\l_s^2x^{d-2}_0x_1^2 \ww \cdots
\ww \sum (-1)^{s+b}\l_s^bx^{d-b}_0x_1^b \right] \\
& \quad + O\big(t^{\binom {b+1}{2}+1} \big)
\end{split}
\]
and (with each $\lambda_{i,s}=1$)
\[
\begin{split}
F_{(1,1)}(t) &= x_0^d \ww (x_0+tx_1)^d \ww (x_0+t^2x_2)^d \ww (x_0+tx_1+t^2x_2)^d \\
& = x_0^d \ww \left(x_0^d + dt x_0^{d-1}x_1 + \binom{d}{2} t^2 x_0^{d-2}x_1^2 + \cdots \right) \\
& \qquad \ww \left(x_0^d + dt^2 x_0^{d-1}x_2 + \cdots \right) \\
& \qquad \ww \left(x_0^d + dtx_0^{d-1}x_1 + t^2\Big(\binom{d}{2}x_0^{d-2}x_1^2 + dx_0^{d-1}x_2\Big) \right. \\
& \qquad \quad \left. + t^3 \Big(\binom{d}{3}x_0^{d-3}x_1^3 + d(d-1)x_0^{d-2}x_1x_2\Big) + \cdots \right) \\
& = t^6 \left( x_0^d \ww d x_0^{d-1}x_1 \ww dx_0^{d-1}x_2 \ww d(d-1)x_0^{d-2}x_1x_2 \right) + O(t^7) .
\end{split}
\]

\begin{theorem}\label{urmonomial}
Let $b_0 \geq b_1 + \cdots + b_n$.
Then $\ur(x_0^{b_0} x_1^{b_1} \cdots x_n^{b_n})=  T_{(b_1,\dots, b_n)}$.
\end{theorem}

Theorem \ref{urmonomial} is an immediate consequence of Theorem \ref{urmonomialest} and the following
lemma:

\begin{lemma}Let $\aaa=(a_1\hd a_n)$. Write $\bbb=(a_0,\aaa)$ with $a_0 \geq a_1+\cdots + a_n$.
Then for $a_1+\cdots +a_n \leq \d \leq a_0$,  
$S_{\bbb,\d}$ is independent of $\d$ and in fact $S_{\bbb,\d}= T_{\aaa}$.
\end{lemma}
\begin{proof}
The right hand side $T_{\aaa}$ counts $n$-tuples $(e_1,\dots,e_n)$ such that $0 \leq e_j \leq a_j$.
To each such tuple we associate the $(n+1)$-tuple $(\d-(e_1+\dots+e_n), e_1,\dots,e_n)$.
Since
\[ 0 \leq \d-(a_1+\dots+a_n) \leq \d-(e_1+\dots+e_n) \leq \d \leq a_0, \]
  this is one of the tuples counted by the left hand side $S_{\bbb,\d}$,
establishing a bijection between the sets counted by $S_{\bbb,\d}$ and $T_{\aaa}$.
\end{proof}

In particular, 
\begin{corollary}\label{cor: simple monomial}
Write $d=a+n$, and consider the monomial $\phi=x_0^{a}x_1 \cdots x_n$.
If $a \geq n$, then 
$\ur(x_0^{a}x_1 \cdots x_n)=2^n$.
Otherwise,  
\[
  \binom{n}{\lfloor \frac{d}{2} \rfloor-a} + \binom{n}{\lfloor \frac{d}{2} \rfloor-a+1} + \cdots +
    \binom {n}{\lfloor \frac{d}{2} \rfloor}
  \leq \ur (x_0^{a}x_1\cdots x_n)\leq 2^n .
\]
\end{corollary}
\begin{proof}
The right hand inequality follows as  $T_{(1,\dots,1)} = 2^n$.
To see the left hand inequality, 
for $0 \leq k \leq a$, let $e = \lfloor \frac{d}{2} \rfloor - a + k$.
Then $\binom{n}{e}$ is the number of monomials of the form
$x_0^{\lfloor d/2 \rfloor - e} x_{i_1} \cdots x_{i_e}$, $1 \leq i_1 < \cdots < i_e \leq n$ and
  $S_{(a,1,\dots,1), \lfloor \frac{d}{2} \rfloor}$ is precisely the total number
of all such monomials for all values of $e$.
\end{proof}

\begin{proposition}\label{hplaneprodprop}
\[
  \binom{n}{\lfloor n/2 \rfloor} + \lceil n/2 \rceil - 1 \leq R(x_1 \cdots x_n) \leq 2^{n-1} ,
\]
\[
  \binom{n}{\lfloor n/2 \rfloor}   \leq \ur (x_1 \cdots x_n) \leq 2^{n-1} .
\]
\end{proposition}

\begin{proof} Write $\phi=x_1\cdots x_n$.
First,
\[
  \phi = \frac{1}{2^{n-1} n!} \sum_{\epsilon \in \{-1,1\}^{n-1}} (x_1 + \epsilon_1 x_2 + \dots + \epsilon_{n-1} x_n)^n \epsilon_1 \cdots \epsilon_{n-1} ,
\]
a sum with $2^{n-1}$ terms, so $R(\phi) \leq 2^{n-1}$.

Now, for $1 \leq a \leq n-1$, the image of $\phi_{a,n-a}$ is spanned by the monomials $x_{i_1} \cdots x_{i_a}$,
$1 \leq i_1 < \dots < i_a \leq n$.
So $\rank \phi_{a,n-a} = \binom{n}{a}$.
Thus   $\ur(\phi) \geq \binom{n}{\lfloor n/2 \rfloor}$.
The set $\Sigma_a$ consists of those points $p \in \BP W^{*} \cong \pp{n-1}$
at which (at least) $a+1$ of the coordinate functions vanish.
So $\dim \Sigma_a = n-a-2$.
Therefore $R(\phi) > \binom{n}{a} + n-a-2$, for $1 \leq a \leq n-1$.
  This quantity is maximized at $a = \lfloor n/2 \rfloor$.
\end{proof}

To give a sense of how these bounds behave, we illustrate with the following table for bounds on  the ranks
and border ranks of $x_1\cdots x_n$.
\[
  \begin{array}{l | ccccc ccccc}
    n & 1 & 2 & 3 & 4 & 5 & 6 & 7 & 8 & 9 & 10 \\
    \hline
    \text{upper bound for $R(x_1\cdots x_n)$}                     & 1 & 2 & 4 & 8 & 16 & 32 & 64 & 128 & 256 & 512 \\
    \text{lower bound for $R(x_1\cdots x_n)$}               & 1 & 2 & 4 & 7 & 12 & 22 & 38 & 73 & 130 & 256 \\
    \text{lower bound for $\ur(x_1\cdots x_n)$} & 1 & 2 & 3 & 6 & 10 & 20 & 35 & 70 & 126 & 252 \\
  \end{array}
\]

For $n<4$ the upper and lower bounds agree. Here is the next case:

\begin{proposition}\label{n4prop} $R(x_1x_2x_3x_4)=8$.
\end{proposition}
\begin{proof}
Suppose $R(x_1 x_2 x_3 x_4) = 7$.
Write $\phi = x_1x_2x_3x_4 = \eta_1^4 + \cdots + \eta_7^4$ with  the $[ \eta_i ] \in \BP W$ distinct points.
Let $L = \{ p\in S^2W^*\mid p(\eta_i)=0,\ i=1\hd 7\}$, so
$\BP L \subset \BP \Rker \phi_{2,2}$.
We have $\dim \BP L \geq \dim\BP S^2W^* - 7 = 2$.
On the other hand, $\BP L$ is contained in $\BP \Rker \phi_{2,2}$ and disjoint
from $\BP \Rker \phi_{2,2} \cap v_2(\BP W^{*}) \cong \Sigma_2$, so $\dim \BP L \leq 2$
(as in the proof of Theorem~\ref{sigmathm}).

We will show that there are six reducible quadrics in $\BP L$,
and they restrict the $\eta_i$ in such a way to imply a contradiction.
 
By Theorem~\ref{sigmathm}, for all $\lambda \neq 0$,
\[
  R(\phi - \lambda x_1^4) \geq \rank (\phi - \lambda x_1^4)_{2,2} + \dim \Sigma_2(\phi - \lambda x_1^4) +1 = 7 + 0 + 1
\]
where $\rank(\phi - \lambda x_1^4)_{2,2}=7$ because the image of $(\phi - \lambda x_1^4)_{2,2}$
is spanned by the $7$ elements
$x_1^2$, $x_1 x_2$, $x_1 x_3, \dots , x_3 x_4 $.
If one of the $\eta_i$ were (a scalar multiple of) $x_1$ then we would have
$R(\phi - \lambda x_1^4) \leq R(\phi) - 1 < 7$.
By the same argument for $x_2, \dots, x_4$, all $11$ of the points $[x_i], [\eta_j]$ are distinct.

Let $\alpha_1, \dots, \alpha_4$ be the dual basis of $W^*$ to $x_1,\dots,x_4$.
Then $\BP \Rker \phi_{2,2}$ is clearly spanned by
$\{ [\alpha_1^2], \dots, [\alpha_4^2] \} = \BP \Rker \phi_{2,2}\cap v_2(\BP W^*)$.
The reducible quadrics in $\BP \Rker \phi_{2,2}$
are precisely the elements $[p \alpha_i^2 + q \alpha_j^2]$, $i \neq j$,
that is, the lines which form the edges of the tetrahedron with vertices at the $[\alpha_i^2]$.
By a dimension count, $L$ intersects these lines.
Since $L$ is a linear subspace, it intersects the tetrahedron at precisely six points,
which are not the vertices.
This shows there are precisely six reducible quadrics passing through the $[\eta_i]$.

Denote them $Q_{12}, \dots, Q_{34}$,
where $Q_{ij}$ spans $L \cap \langle \alpha_i^2, \alpha_j^2 \rangle$.
Up to scaling the $Q_{ij}$, there are constants $b_1,\dots,b_4$
such that $Q_{ij} = b_i \alpha_i^2 - b_j \alpha_j^2$.
(Indeed, writing each $Q_{1j} = \alpha_1^2 - b_j \alpha_j^2$,
$Q_{jk}$ must be a scalar times $Q_{1k} - Q_{1j}$, from which the claim follows.)
The $b_i$ are nonzero, so we may rescale coordinates so each $b_i=1$.

Then up to scalar multiple each $\eta_i = x_1 \pm x_2 \pm x_3 \pm x_4$.
Solving for the coefficients $c_i$ in $x_1 x_2 x_3 x_4 = c_1 \eta_1^4 + \cdots + c_7 \eta_7^4$
shows there are no solutions.
Equivalently, let $\eta_1,\dots,\eta_8$ be all $8$ of the points $x_1 \pm x_2 \pm x_3 \pm x_4$.
There is no solution for $c_i$ in $x_1 x_2 x_3 x_4 = c_1 \eta_1^4 + \cdots + c_8 \eta_8^4$
with one of the $c_i=0$.
\end{proof}

\begin{remark}
The singular quadrics in $\BP \Rker \phi_{2,2}$ are those of the form
$[p \alpha_{i_1}^2 + q \alpha_{i_2}^2 + r \alpha_{i_3}^2]$, where $\{i_1,i_2,i_3\}\subset\{ 1,2,3,4\}$
which correspond to the faces of the tetrahedron spanned by $[\alpha_1^2], \dots, [\alpha_4^2]$.
Each such quadric is singular at $[x_{i_4}]$, where $\{i_1,i_2,i_3,i_4\}=\{ 1,2,3,4\}$.
It would be   interesting to see if considering these singular quadrics, instead of the reducible quadrics,
yields a simpler proof that $R(x_1 x_2 x_3 x_4) > 7$.
\end{remark}

\begin{remark}\label{rem: ranks of monomials}
One can get lower and upper bounds on the ranks of monomials
by Theorem~\ref{sigmathm} and specialization.
The upper bound, for $b_0 \geq \cdots \geq b_n$,
is $R(x_0^{b_0} \cdots x_n^{b_n}) \leq (b_0+1) \cdots (b_{n-1}+1) b_n$.
This follows from considering the polarization-type identity appearing in the proof of Proposition~\ref{hplaneprodprop}
for the product $y_{0,1} \cdots y_{0,b_0} \cdots y_{n,b_n}$ and then specializing each $y_{i,j} \to x_i$.
\end{remark}

\begin{proposition}
$R(x^2yz)=6$ and $\ur(x^2yz)=4$.\label{example x2yz}
\end{proposition}
\begin{proof}
Let $\phi = x^2 yz$.
We have $\ur(\phi)=4$ by Theorem~\ref{urmonomial}.
We have $5\leq R(\phi)\leq 6$ by 
Remark~\ref{rem: ranks of monomials}.
(Explicitly: $R(\phi) \geq 5$ by Theorem~\ref{sigmathm}.
The upper bound comes from $x^2 yz = x^2((y+z)/2)^2 - x^2((y-z)/2)^2$ where each term has the form $a^2 b^2$,
and $R(a^2 b^2)=3$ by Corollary~\ref{corab}.)

We will show that in fact $R(\phi)=6$, following a suggestion provided to us by Bruce Reznick.
Suppose that $R(\phi) = 5$, with $\phi = \eta_1^4 + \dots + \eta_5^4$,
for some distinct $[\eta_i] \in \BP W = \pp{2}$.
Let $L := \BP \{ p\in S^{2}W^* \mid p( \eta_i )=0, 1\leq i\leq 5\}$.
The proof of Theorem~\ref{sigmathm} shows $\dim L = 0$, i.e., $L$ consists of exactly one point,
so the $[\eta_i]$ lie on a unique conic $Q$ in the projective plane.
In particular, no four of the $[\eta_i]$ are collinear.
One checks that $R(x^2yz - \lambda x^4) \geq 5$ by Theorem~\ref{sigmathm}, for all $\lambda$,
and so no $[\eta_i] = [x]$.

The conic $Q$ is an element of $\BP \Rker \phi_{2,2}$,
which one finds is spanned by $\beta^2$ and $\gamma^2$.
Therefore $Q$ factors, $Q = (c \beta - d \gamma)(c \beta + d \gamma)$.
We have $c, d \neq 0$ (or else all five $[\eta_i]$ are collinear).

Therefore exactly three of the $[\eta_i]$ lie on one of the lines of $Q$ and exactly two lie on the other line.
Up to reordering, we have $\eta_i = s_i x + t_i (dy + cz)$ for $i = 1, 2, 3$
and $\eta_i = s_i x + t_i (dy - cz)$ for $i=4,5$.
The subsitution $z \to \frac{-d}{c} y$ takes the equation
\[ \phi = x^2 yz = \eta_1^4 + \dots + \eta_5^4 \]
to
\[ \frac{-d}{c} x^2 y^2 = ( s_1^4 + s_2^4 + s_3^4 ) x^4 + \overline{\eta}_4^4 + \overline{\eta}_5^4 \]
where $\overline{\eta}_4, \overline{\eta}_5$ are linear forms in $x, y$.
Multiplying by scalar factors,
this gives an expression of $x^2y^2 - A x^4$ as a sum of two fourth powers.
But we have $R(x^2 y^2 - A x^4) \geq 3$ for all $A$;
indeed, the symmetric flattening $(x^2y^2 - A x^4)_{2,2}$ has rank $3$ already.

This contradiction shows $R(\phi) > 5$.
\end{proof}

\bibliographystyle{amsplain}
 

\providecommand{\bysame}{\leavevmode\hbox to3em{\hrulefill}\thinspace}
\providecommand{\MR}{\relax\ifhmode\unskip\space\fi MR }
\providecommand{\MRhref}[2]{%
  \href{http://www.ams.org/mathscinet-getitem?mr=#1}{#2}
}
\providecommand{\href}[2]{#2}

\end{document}